\documentclass[10pt,leqno,twoside,draft]{amsart}
\usepackage{amsmath}
\usepackage{amsfonts}
\usepackage{enumerate}
\usepackage{amsthm}
\usepackage{bm}
\usepackage{mathtools}
\usepackage{color}
\makeatletter \@addtoreset{equation}{section} \makeatother
\newcommand{\eref}[1]{(\ref{#1})}
\newcommand{\tref}[1]{Theorem \ref{#1}}
\newcommand{\pref}[1]{Proposition \ref{#1}}
\newcommand{\lref}[1]{Lemma \ref{#1}}
\newcommand{\cref}[1]{Corollary \ref{#1}}
\newcommand{\rref}[1]{Remark \ref{#1}}
\newcommand{\sref}[1]{Section \ref{#1}}

\newenvironment{spmatrix}{ \left(\begin{smallmatrix} }{ \end{smallmatrix}\right) }
\theoremstyle{plain} \newtheorem{thm}{Theorem}[section] \newtheorem{lem}{Lemma}[section] \newtheorem{prop}{Proposition}[section] \newtheorem{cor}{Corollary}[section]
\theoremstyle{definition} \newtheorem{rem}{Remark}[section] \newtheorem{defi}{Definition}[section]
\title[$L^\infty$-error estimates of FEM]{$L^\infty$- and $W^{1,\infty}$-error estimates of linear finite element method for Neumann boundary value problems in a smooth domain}
\author[Takahito Kashiwabara]{Takahito Kashiwabara}
\address{Graduate School of Mathematical Sciences, The University of Tokyo, 3-8-1 Komaba, Meguro, 153-8914 Tokyo, Japan}
\email{tkashiwa@ms.u-tokyo.ac.jp}
\author[Tomoya Kemmochi]{Tomoya Kemmochi}
\address{Department of Applied Physics, Graduate School of Engineering, Nagoya University, Furo-cho, Chikusa-ku, 464-8603 Nagoya, Japan}
\email{kemmochi@na.nuap.nagoya-u.ac.jp}
\thanks{This work was supported by JSPS Grant-in-Aid for Young Scientists B (No.\ 17K14230).}

\begin{document}
\begin{abstract}
	Pointwise error analysis of the linear finite element approximation for $-\Delta u + u = f$ in $\Omega$, $\partial_n u = \tau$ on $\partial\Omega$, where $\Omega$ is a bounded smooth domain in $\mathbb R^N$, is presented.
	We establish $O(h^2|\log h|)$ and $O(h)$ error bounds in the $L^\infty$- and $W^{1,\infty}$-norms respectively, by adopting the technique of regularized Green's functions combined with local $H^1$- and $L^2$-estimates in dyadic annuli.
	Since the computational domain $\Omega_h$ is only polyhedral, one has to take into account non-conformity of the approximation caused by the discrepancy $\Omega_h \neq \Omega$.
	In particular, the so-called Galerkin orthogonality relation, utilized three times in the proof, does not exactly hold and involves domain perturbation terms (or boundary-skin terms), which need to be addressed carefully.
	A numerical example is provided to confirm the theoretical result.
\end{abstract}
\maketitle

\section{Introduction} 
We consider the following Poisson equation with a non-homogeneous Neumann boundary condition:
\begin{equation} \label{eq: Poisson}
	-\Delta u + u = f \quad\text{in}\quad\Omega, \qquad \partial_n u = \tau \quad\text{on}\quad \Gamma := \partial\Omega,
\end{equation}
where $\Omega \subset \mathbb R^N$ is a bounded domain with a smooth boundary $\Gamma$ of $C^\infty$-class, $f$ is an external force, $\tau$ is a prescribed Neumann data, and $\partial_n$ means the directional derivative with respect to the unit outward normal vector $n$ to $\Gamma$.
The linear (or $P_1$) finite element approximation to \eref{eq: Poisson} is quite standard.
Given an approximate polyhedral domain $\Omega_h$ whose vertices lie on $\Gamma$, one can construct a triangulation $\mathcal T_h$ of $\Omega_h$, build a finite dimensional space $V_h$ consisting of piecewise linear functions, and seek for $u_h \in V_h$ such that
\begin{equation} \label{eq: FEM}
	(\nabla u_h, \nabla v_h)_{\Omega_h} + (u_h, v_h)_{\Omega_h} = (\tilde f, v_h)_{\Omega_h} + (\tilde\tau, v_h)_{\Gamma_h} \qquad \forall v_h \in V_h,
\end{equation}
where $\Gamma_h := \partial\Omega_h$, and $\tilde f$ and $\tilde\tau$ denote extensions of $f$ and $\tau$, respectively.
Then, the main result of this paper is the following pointwise error estimates in the $L^\infty$- and $W^{1,\infty}$-norms:
\begin{equation} \label{eq: main result}
	\begin{aligned}
		\|\tilde u - u_h\|_{L^\infty(\Omega_h)} &\le Ch^2|\log h| \, \|u\|_{W^{2, \infty}(\Omega)}, \\
		\|\tilde u - u_h\|_{W^{1,\infty}(\Omega_h)} &\le Ch \, \|u\|_{W^{2, \infty}(\Omega)},
	\end{aligned}
\end{equation}
where $h$ denotes the mesh size of $\mathcal T_h$, and $\tilde u$ is an arbitrary extension of $u$ (of course, the way of extension must enjoy some stability, cf.\ \sref{sec2.3} below).

Regarding pointwise error estimates of the finite element method, there have been many contributions since 1970s (for example, see the references in \cite{Sch98}), and, consequently, standard methods to derive them are now available.
The strategy of those methods is briefly explained as follows.
By duality, analysis of $L^\infty$- or $W^{1,\infty}$-error of $u - u_h$ may be reduced to that of $W^{1,1}$-error between a regularized Green's function $g$, with singularity near $x_0 \in \Omega$, and its finite element approximation $g_h$.
To deal with $\|\nabla (g - g_h)\|_{L^1(\Omega)}$ in terms of energy norms, it is estimated either by $\sum_{j=0}^J d_j^{N/2}\|\nabla (g - g_h)\|_{L^2(\Omega \cap A_j)}$ or by $\|\sigma^{N/2} \nabla (g - g_h)\|_{L^2(\Omega)}$,
where $\{d_j\}_{j=0}^J$ are radii of dyadic annuli $A_j$ shrinking to $x_0$ with the minimum $d_J = Kh$, whereas $\sigma(x) := (|x - x_0|^2 + \kappa h^2)^{1/2}$.
The two strategies may be regarded as using discrete and continuous weights, respectively, and basically lead to the same results.
In this paper, we employ the first approach, in which scaling heuristics seem to work easier (in the second approach one actually needs to introduce an artificial parameter $\lambda\in(0,1)$ to avoid singular integration, which makes the weighted norm slightly complicated, cf.\ Remark 8.4.4 of \cite{BrSc07}).

The main difficulty of our problem lies in the non-conformity $V_h \not\subset H^1(\Omega)$ arising from the discrepancy $\Omega_h \neq \Omega$ and $\Gamma_h\neq\Gamma$, which we refer to as \emph{domain perturbation}.
In fact, the so-called Galerkin orthogonality relation (or consistency) does not exactly hold, and hence the standard methodology of error estimate cannot be directly applied. 
This issue was already considered in classical literature (see \cite[Section 4.4]{StFi73} or \cite[Section 4.4]{Cia78}) as long as energy-norm (i.e.\ $H^1$) error estimates for a Dirichlet problem are concerned.
However, there are much fewer studies of error analysis in other norms or for other boundary value problems, which take into account domain perturbation.
For example, \cite{BaEl88, Cer83} gave optimal $L^2$-error estimates for a Robin boundary value problem.

As for pointwise error estimates, the issue of domain perturbation was mainly treated only for a homogeneous Dirichlet problem in a convex domain.
In this case, one has a conforming approximation $V_h \cap H^1_0(\Omega_h) \subset H^1_0(\Omega)$ with the aid of the zero extension, which makes error analysis simpler.
This situation was studied for elliptic problems in \cite{BTW02, ScWa82} and for parabolic ones in \cite{Gei06, ThWa00}.
Although an idea to treat $\Omega_h \not\subset \Omega$ in the case of $L^\infty$-analysis is found in \cite[p.\ 2]{ScWa82}, it does not seem to be directly applicable to $W^{1,\infty}$-analysis or to Neumann problems.
In \cite{Gei06, Sch98, STW98}, they considered Neumann problems in a smooth domain assuming that triangulations exactly fit a curved boundary, where one need not take into account domain perturbation.
This assumption, however, excludes the use of usual Lagrange finite elements.
The $P_2$-isoparametric finite element analysis for a Dirichlet problem $(N=2)$ was shown in \cite{Wah78}, where the rate of convergence $O(h^{3-\epsilon})$ in the $L^\infty$-norm was obtained.

The aim of this paper is to present pointwise error analysis of the finite element method taking into account full non-conformity caused by domain perturbation.
To focus on its effect, we choose the simplest PDE \eref{eq: Poisson} and the simplest finite element, i.e., the $P_1$-element, as a model case.
However, since a non-homogeneous Neumann boundary condition is considered in a non-convex domain, no simplification mentioned above is available.
Our conclusion is already stated in \eref{eq: main result}, which implies that domain perturbation does not affect the rate of convergence in the $L^\infty$- and $W^{1,\infty}$-norms known for the case $\Omega_h = \Omega$, when $P_1$-elements are used to approximate both a curved domain and a solution.
This leads us to conjecture that the use of general isoparametric elements would keep the (quasi-)optimal rate of convergence in non-energy norms as well.

Finally, let us make a comment concerning the opinion that the issue of $\Omega_h \neq \Omega$ is similar to that of numerical integration (see \cite[p.\ 1356]{STW98}).
As mentioned in the same paragraph there, if a computational domain is extended (or transformed) to include $\Omega$ and a restriction (or transformation) operator to $\Omega$ is applied, then one can disregard the effect of domain perturbation (higher-order schemes based on such a strategy are proposed e.g.\ in \cite{CoSo12}).
On the other hand, since implementing such a restriction operator precisely for general domains is non-trivial in practical computation, some approximation of geometric information for $\Omega$ should be incorporated in the end.
Thereby one needs to more or less deal with domain perturbation in error analysis, and, in our opinion, its rigorous treatment would be quite different from that of numerical integration.

The organization of this paper is as follows.
Basic notations are introduced in \sref{sec2}, together with boundary-skin estimates and a concept of dyadic decomposition.
In \sref{sec3}, we present the main result (\tref{main theorem}) and reduce its proof to $W^{1,1}$-error estimate of $g - g_h$.
The weighted $H^1$- and $L^2$-error estimates of $g - g_h$ are shown in Sections \ref{sec4} and \ref{sec5}, respectively, which are then combined to complete the proof of \tref{main theorem} in \sref{sec6}.
A numerical example is given to confirm the theoretical result in \sref{sec7}.
Throughout this paper, $C>0$ will denote generic constants which may be different at each occurrence; its dependency (or independency) on other quantities will often be mentioned as well.
However, when it appears with sub- or super-scripts (e.g., $C_{0E}, C'$), we do not treat it as generic.

\section{Preliminaries} \label{sec2}
\subsection{Basic notation} \label{sec2.1}
Recall that $\Omega \subset \mathbb R^N$ is a bounded $C^\infty$-domain.
We employ the standard notation of the Lebesgue spaces $L^p(\Omega)$, Sobolev spaces $W^{s,p}(\Omega)$ (in particular, $H^s(\Omega) := W^{s, 2}(\Omega)$), and H\"older spaces $C^{m,\alpha}(\overline\Omega)$.
Throughout this paper we assume the regularity $u \in W^{2,\infty}(\Omega)$ for \eref{eq: Poisson}, which is indeed true if $f \in C^\alpha(\overline\Omega)$ and $\tau \in C^\alpha(\overline\Gamma)$ for some $\alpha \in (0,1)$.

Given a bounded domain $D \subset \mathbb R^N$, both of the $N$-dimensional Lebesgue measure of $D$ and the $(N-1)$-dimensional surface measure of $\partial D$ are simply denoted by $|D|$ and $|\partial D|$, as far as there is no fear of confusion.
Furthermore, we let $(\cdot, \cdot)_D$ and $(\cdot, \cdot)_{\partial D}$ be the $L^2(D)$- and $L^2(\partial D)$-inner products, respectively, and define the bilinear form
\begin{equation*}
	a_D(u, v) := (\nabla u, \nabla v)_D + (u, v)_D, \qquad u,v \in H^1(D),
\end{equation*}
which is simply written as $a(u, v)$ when $D = \Omega$, and as $a_h(u, v)$ when $D = \Omega_h$ (to be defined below).

Letting $\Omega_h$ be a polyhedral domain, we consider a family of triangulations $\{\mathcal T_h\}_{h\downarrow0}$ of $\Omega_h$ which consist of closed and mutually disjoint simplices.
We assume that $\{\mathcal T_h\}_{h\downarrow0}$ is quasi-uniform, that is, every $T \in \mathcal T_h$ contains (resp.\ is contained in) a ball with the radius $ch$ (resp.\ $h$), where $h := \max_{T\in\mathcal T_h} h_T$ with $h_T := \operatorname{diam}T$.
The boundary mesh on $\Gamma_h := \partial\Omega_h$ inherited from $\mathcal T_h$ is denoted by $\mathcal S_h$, namely, $\mathcal S_h = \{S \subset \Gamma_h \,|\, \text{$S$ is an $(N-1)$-dimensional face of some $T \in \mathcal T_h$} \}$.
We then assume that the vertices of every $S \in \mathcal S_h$ belong to $\Gamma$, that is, $\Gamma_h$ is essentially a linear interpolation of $\Gamma$.

The linear (or $P_1$) finite element space $V_h$ is given in a standard manner, i.e.,
\begin{equation*}
	V_h = \{v_h \in C(\overline\Omega_h) \,:\, v_h|_T \in P_1(T) \quad \forall T\in \mathcal T_h \},
\end{equation*}
where $P_k(T)$ stands for the polynomial functions defined in $T$ with degree $\le k$.

Let us recall a well-known result of an interpolation operator (also known as a local regularization operator) $\mathcal I_h : H^1(\Omega_h) \to V_h$ satisfying the following property (see \cite[Section 4.8]{BrSc07}):
\begin{equation*}
	\|\nabla^k(v - \mathcal I_h v)\|_{L^p(T)} \le C_{\mathcal I} h_T^{m-k} \|\nabla^m v\|_{L^p(M_T)} \qquad k=0,1, \, m=1,2, \, v\in W^{m,p}(\Omega_h),
\end{equation*}
where $M_T := \bigcup\{T' \in \mathcal T_h \,:\, T'\cap T \neq \emptyset \}$ is a macro-element of $T \in \mathcal T_h$.
The constant $C_{\mathcal I}$ depends on $c,k,m,p$ and on a reference element; especially it is independent of $v$ and $h_T$.
We also use the trace estimate
\begin{equation*}
	\|v\|_{L^2(\Gamma_h)} \le C \|v\|_{L^2(\Omega_h)}^{1/2} \|v\|_{H^1(\Omega_h)}^{1/2},
\end{equation*}
where $C$ depends on the $C^{0,1}$-regularity of $\Omega_h$ and thus it is uniformly bounded by that of $\Omega$ for $h\le 1$.

\subsection{Boundary-skin estimates} \label{sec2.2}
To examine the effects due to the domain discrepancy $\Omega_h \neq \Omega$, we introduce a notion of tubular neighborhoods $\Gamma(\delta) := \{x\in\mathbb R^N \,:\, \operatorname{dist}(x, \Gamma) \le \delta\}$.
It is known that (see \cite[Section 14.6]{GiTr98}) there exists $\delta_0>0$, which depends on the $C^{1,1}$-regularity of $\Omega$, such that each $x \in \Gamma(\delta_0)$ admits a unique representation
\begin{equation*}
	x = \bar x + tn(\bar x), \qquad \bar x\in\Gamma, \, t\in[-\delta_0, \delta_0].
\end{equation*}
We denote the maps $\Gamma(\delta_0)\to \Gamma$; $x\mapsto\bar x$ and $\Gamma(\delta_0)\to \mathbb R$; $x\mapsto t$ by $\pi(x)$ and $d(x)$, respectively
(actually, $\pi$ is an orthogonal projection to $\Gamma$ and $d$ agrees with the signed-distance function).
The regularity of $\Omega$ is inherited to that of $\pi$, $d$, and $n$ (cf.\ \cite[Section 7.8]{DeZo11}).

In \cite[Section 8]{KOZ16} we proved that $\pi|_{\Gamma_h}$ gives a homeomorphism (and piecewisely a diffeomorphism) between $\Gamma$ and $\Gamma_h$ provided $h$ is sufficiently small,
taking advantage of the fact that $\Gamma_h$ can be regarded as a linear interpolation of $\Gamma$ (recall the assumption on $\mathcal S_h$ mentioned above).
If we write its inverse map $\pi^* : \Gamma\to\Gamma_h$ as $\pi^*(x) = \bar x + t^*(\bar x) n(\bar x)$, then $t^*$ satisfies the estimates $\|\nabla_\Gamma^k t^*\|_{L^\infty(\Gamma)} \le C_{kE}h^{2-k}$ for $k=0,1,2$, where $\nabla_\Gamma$ means the surface gradient along $\Gamma$ and where the constant depends on the $C^{1,1}$-regularity of $\Omega$.
This in particular implies that $\Omega_h\triangle\Omega := (\Omega_h\setminus\Omega) \cup (\Omega\setminus\Omega_h)$ and $\Gamma_h \cup \Gamma$ are contained in $\Gamma(\delta)$ with $\delta := C_{0E}h^2$.
We refer to $\Omega_h\triangle\Omega$, $\Gamma(\delta)$ and their subsets as \emph{boundary-skin layers} or more simply as \emph{boundary skins}.

Furthermore, we know from \cite[Section 8]{KOZ16} the following boundary-skin estimates:
\begin{equation} \label{eq: boundary-skin estimates}
	\begin{aligned}
		\left| \int_\Gamma f\,d\gamma - \int_{\Gamma_h} f\circ\pi\,d\gamma_h \right| &\le C\delta \|f\|_{L^1(\Gamma)}, \\
		\|f\|_{L^p(\Gamma(\delta))} &\le C(\delta^{1/p} \|f\|_{L^p(\Gamma)} + \delta \|\nabla f\|_{L^p(\Gamma(\delta))}), \\
		\|f - f\circ\pi\|_{L^p(\Gamma_h)} &\le C\delta^{1-1/p} \|\nabla f\|_{L^p(\Gamma(\delta))},
	\end{aligned}
\end{equation}
where one can replace $\|f\|_{L^1(\Gamma)}$ in \eref{eq: boundary-skin estimates}$_1$ by $\|f\|_{L^1(\Gamma_h)}$.
As a version of \eref{eq: boundary-skin estimates}$_2$, we also need
\begin{equation} \label{eq: boundary-skin estimate 2}
	\|f\|_{L^p(\Omega_h\setminus\Omega)} \le C(\delta^{1/p} \|f\|_{L^p(\Gamma_h)} + \delta \|\nabla f\|_{L^p(\Omega_h\setminus\Omega)}),
\end{equation}
whose proof will be given in \lref{lem: proof of (2.2)}.
Finally, denoting by $n_h$ the outward unit normal to $\Gamma_h$, we notice that its error compared with $n$ is estimated as $\|n\circ\pi - n_h\|_{L^\infty(\Gamma_h)} \le Ch$ (see \cite[Section 9]{KOZ16}).

\subsection{Extension operators} \label{sec2.3}
We let $\tilde\Omega := \Omega\cup\Gamma(\delta) = \Omega_h\cup\Gamma(\delta)$ with $\delta = C_{0E}h^2$ given above.
For $u \in W^{2,\infty}(\Omega), f \in L^\infty(\Omega)$, and $\tau \in L^\infty(\Gamma)$, we assume that there exist extensions $\tilde u \in W^{2,\infty}(\tilde\Omega), \tilde f \in L^\infty(\tilde\Omega)$, and $\tilde\tau \in L^\infty(\tilde\Omega)$, respectively, which are stable in the sense that the norms of the extended quantities can be controlled by those of the original ones, e.g., $\|\tilde u\|_{W^{2,\infty}(\tilde\Omega)} \le C\|u\|_{W^{2,\infty}(\Omega)}$.
We emphasize that \eref{eq: Poisson} would not hold any longer in the extended region $\tilde\Omega \setminus \overline\Omega$.

We also need extensions whose behavior in $\Gamma(\delta)\setminus\Omega$ can be completely described by that in $\Gamma(c\delta) \cap \Omega$ for some constant $c>0$.
To this end we introduce an extension operator $P : W^{k,p}(\Omega) \to W^{k,p}(\tilde\Omega) \, (k=0,1,2, p \in [1,\infty])$ as follows.
For $x \in \Omega\setminus\Gamma(\delta)$ we let $Pf(x) = f(x)$; for $x = \bar x + tn(\bar x) \in \Gamma(\delta)$ we define
\begin{equation*}
	Pf(\bar x + tn(\bar x)) = \begin{cases} f(\bar x + t n(\bar x)) & (-\delta_0\le t<0), \\ 3f(\bar x - tn(\bar x)) - 2f(\bar x - 2t n(\bar x)) & (0\le t\le\delta_0), \end{cases} \qquad \bar x \in \Gamma.
\end{equation*}
\begin{prop} \label{prop: stability of extension}
	The extension operator $P$ satisfies the following stability condition:
	\begin{align*}
		\|Pf\|_{W^{k, p}(\Gamma(\delta))} &\le C\|f\|_{W^{k, p}(\Omega \cap \Gamma(2\delta))} \quad (k=0,1,2), \quad p\in[1,\infty],
	\end{align*}
	where $C$ is independent of $\delta$ and $f$.
\end{prop}
The proof of this proposition will be given in \tref{thm: extension operator}.

\subsection{Dyadic decomposition}
We introduce a dyadic decomposition of a domain according to \cite{Sch98}.
Let $B(x_0; r) = \{x \in \mathbb R^N \,:\, |x - x_0| \le r\}$ and $A(x_0; r, R) = \{x \in \mathbb R^N \,:\, r\le |x - x_0|\le R\}$ denote a closed ball and annulus in $\mathbb R^N$ respectively.
\begin{defi}
	For $x_0 \in \mathbb R^N, d_0>0, J \in \mathbb N_{\ge0}$, the family of sets $\mathcal A(x_0, d_0, J) = \{A_j\}_{j=0}^J$ defined by
	\begin{equation*}
		A_0 = B(x_0; d_0), \quad A_j = A(x_0; d_{j-1}, d_j), \quad d_j = 2^j d_0 \quad (j=1,\dots,J)
	\end{equation*}
	is called the \emph{dyadic $J$ annuli with the center $x_0$ and the initial stride $d_0$}.
\end{defi}
With a center and an initial stride specified, one can assign to a given domain a unique decomposition by dyadic annuli as follows.
\begin{lem} \label{lem: dyadic covering}
	For a bounded domain $\Omega \subset \mathbb R^N$, let $x_0 \in \Omega$, $0<d_0<\operatorname{diam}\Omega$, and $J$ be the smallest integer that is greater than $J' := \frac{\log( \operatorname{diam}\Omega/d_0 )}{\log2}$.
	Then we have $\overline\Omega \subset \bigcup\mathcal A(x_0, d_0, J)$.
\end{lem}
\begin{proof}
	Since $2^{J'}d_0 = \operatorname{diam}\Omega$ and $J'<J\le J'+1$, one has $\operatorname{diam}\Omega<d_J\le 2\operatorname{diam}\Omega$.
	For arbitrary $x\in\Omega$ we see that $|x - x_0| \le \operatorname{diam}\Omega<d_J$, which implies $\overline\Omega \subset B(x_0; d_J) = \bigcup \mathcal A(x_0, d_0, J)$.
\end{proof}
\begin{defi}
	We define the \emph{decomposition of $\Omega$ into dyadic annuli with the center $x_0$ and the initial stride $d_0$} by $\mathcal A_\Omega(x_0, d_0) = \{\Omega\cap A_j\}_{j=0}^J$, where $\{A_j\}_{j=0}^J = \mathcal A(x_0, d_0, J)$ are the dyadic annuli given in \lref{lem: dyadic covering}.
	We also use the terminology \emph{dyadic decomposition} for abbreviation.
\end{defi}

For $\mathcal A(x_0, d_0, J) = \{A_j\}_{j=0}^J$ and $s \in [0,1]$, we consider expanded annuli $\mathcal A^{(s)}(x_0, d_0, J) = \{A_j^{(s)}\}_{j=0}^J$, where
\begin{equation*}
	A_j^{(s)} := A(x_0; (1- \textstyle\frac{s}2)d_{j-1}, (1+s)d_j).
\end{equation*}
In particular, for $s=1$ one has $A_j^{(1)} = A_{j-1}\cup A_j\cup A_{j+1}$ where we set $A_{-1} := \emptyset$ and $A_{J+1} := A(x_0; d_J, d_{J+1})$ with $d_{J+1} := 2d_J$.

We collect some basic properties of weighted $L^p$-norms defined on a dyadic decomposition.
\begin{lem}
	For a dyadic decomposition $\mathcal A_\Omega(x_0, d_0) = \{ \Omega\cap A_j \}_{j=0}^J$ of $\Omega$ and $p\in[1,\infty]$, the following estimates hold:
	\begin{align}
		\|f\|_{L^1(\Omega)} &\le \alpha_N^{1/p'} \sum_{j=0}^J d_j^{N/p'} \|f\|_{L^p(\Omega\cap A_j)} \label{eq: L1 can be bdd by weighted Lp}, \\
		\sum_{j=0}^J d_j^{N/p'} \|f\|_{L^p(\Omega\cap A_j^{(1)})} &\le (2^{N/p'} + 1 + 2^{-N/p'}) \sum_{j=0}^J d_j^{N/p'} \|f\|_{L^p(\Omega\cap A_j)}. \label{eq: reduction of expanded annuli to usual ones}
	\end{align}
	Here, $\alpha_N = \frac{2\pi^{N/2}}{N\Gamma(N/2)}$ means the volume of the $N$-dimensional unit ball and $p' = p/(p-1)$.
\end{lem}
\begin{proof}
	It follows from the H\"older inequality that
	\begin{align*}
		\|f\|_{L^1(\Omega)} = \sum_{j=0}^J \|f\|_{L^1(\Omega\cap A_j)} \le \sum_{j=0}^J |A_j|^{1/p'} \|f\|_{L^p(\Omega\cap A_j)},
	\end{align*}
	which combined with $|A_j| = (1 - 2^{-N})d_j^N \alpha_N$ yields \eref{eq: L1 can be bdd by weighted Lp}.
	The estimate \eref{eq: reduction of expanded annuli to usual ones} follows from the fact that
	\begin{equation*}
		\|f\|_{L^p(\Omega\cap A_j^{(1)})} \le \|f\|_{L^p(\Omega\cap A_{j-1})} + \|f\|_{L^p(\Omega\cap A_j)} + \|f\|_{L^p(\Omega\cap A_{j+1})},
	\end{equation*}
	together with $\Omega \cap A_{-1} = \Omega \cap A_{J+1} = \emptyset$.
\end{proof}

We need the following lemmas to take care of consistency between annuli in $\mathcal A_\Omega(x_0, d_0)$ and triangles in $\mathcal T_h$.
\begin{lem} \label{lem: consistency of annuli and triangulation}
	Let $\mathcal A_\Omega(x_0, d_0) = \{\Omega \cap A_j\}_{j=0}^J$ be a dyadic decomposition of $\Omega$ with $d_0 \in [16h, 1]$ and $s\in[0, 3/4]$.
	
	(i) If $T\in\mathcal T_h$ satisfies $T \cap A_j^{(s)} \neq \emptyset$ then $M_T \subset A_j^{(s+1/4)}$, where $M_T$ is the macro element of $T$.
	
	(ii) If $T\in\mathcal T_h$ satisfies $T \setminus A_j^{(s+1/4)} \neq \emptyset$ then $M_T \subset (A_j^{(s)})^c$, where the superscript ``$c$'' means the complement set in $\mathbb R^N$.
\end{lem}
\begin{proof}
	We only prove (i) since item (ii) can be shown similarly.
	Let $x \in M_T$ be arbitrary.  By assumption there exists $x' \in T \cap A_j^{(s)}$; in particular, $(1-s/2)d_{j-1} \le |x' - x_0| \le (1+s)d_j$.
	Also, by definition of $M_T$, $|x - x'| \le 2h$.
	Then we have $(7/8 - s/2)d_{j-1} \le |x - x_0| \le (5/4 + s)d_j$ as a result of triangle inequalities, which implies $x \in A_j^{(s+1/4)}$.
\end{proof}

\begin{cor} \label{cor: support of interpolant}
	Under the assumption of \lref{lem: consistency of annuli and triangulation}, let $v \in H^1(\overline\Omega_h)$ satisfy $\operatorname{supp}v \subset A_j^{(s)}$.
	Then we have $\operatorname{supp}\mathcal I_hv \subset A_j^{(s+1/4)}$.
\end{cor}
\begin{proof}
	It suffices to show $\mathcal I_hv(x) = 0$ for all $x \in \Omega_h\setminus A_j^{(s+1/4)}$.
	In fact, since there exists $T \in \mathcal T_h$ such that $x \in T$, one has $M_T \cap A_j^{(s)} = \emptyset$ as a result of \lref{lem: consistency of annuli and triangulation}(ii).
	Hence $v|_{M_T} = 0$, so that $\mathcal I_hv|_T = 0$.
\end{proof}

Finally, note that for any dyadic decomposition $\mathcal A_\Omega(x_0, d_0)$ we have
\begin{equation} \label{eq: geometric sequence}
	\sum_{j=0}^J d_j^{\beta} \le \begin{cases} Cd_J^\beta & (\beta>0), \\ C(1 + |\log d_0|) & (\beta = 0), \\ Cd_0^\beta & (\beta<0), \end{cases}
\end{equation}
where $C = C(N, \Omega, \beta)$ is independent of $x_0, d_0$, and $J$.
Moreover, since $d_j \le d_J \le 2\operatorname{diam}\Omega$, one has
\begin{equation*}
	d_j^\alpha + d_j^\beta \le Cd_j^{\min\{\alpha, \beta\}}, \qquad 0\le j\le J, \; \alpha,\beta\in\mathbb R, \; C=C(N, \Omega, \alpha, \beta),
\end{equation*}
which will not be emphasized in the subsequent arguments.

\section{Main theorem and its reduction to $W^{1,1}$-analysis} \label{sec3}
Let us state the main result of this paper.
\begin{thm} \label{main theorem}
	Let $u\in W^{2,\infty}(\Omega)$ and $u_h \in V_h$ be the solutions of \eref{eq: Poisson} and \eref{eq: FEM} respectively.
	Then there exists $h_0>0$ such that for all $h \in (0, h_0]$ and $v_h \in V_h$ we have
	\begin{align*}
		\|\tilde u - u_h\|_{L^\infty(\Omega_h)} &\le Ch|\log h| \, \|\tilde u - v_h\|_{W^{1,\infty}(\Omega_h)} + Ch^2|\log h| \, \|u\|_{W^{2,\infty}(\Omega)}, \\
		\|\tilde u - u_h\|_{W^{1,\infty}(\Omega_h)} &\le C \|\tilde u - v_h\|_{W^{1,\infty}(\Omega_h)} + Ch \, \|u\|_{W^{2,\infty}(\Omega)},
	\end{align*}
	where $C$ is independent of $h, u$, and $v_h$.
\end{thm}
\begin{rem} \label{rem1 in sec3}
	(i) By taking $v_h = \mathcal I_h \tilde u$, we immediately obtain \eref{eq: main result}.

	(ii) The factor $h \|\tilde u - v_h\|_{W^{1,\infty}(\Omega_h)}$ in the $L^\infty$-estimate could be replaced by $\|\tilde u - v_h\|_{L^\infty(\Omega_h)}$ (cf.\ \cite[p.\ 889]{Sch98}), which will be discussed elsewhere.
	
	(iii) The $O(h^2|\log h|)$ and $O(h)$ contributions in the $L^\infty$- and $W^{1,\infty}$-error estimates would not be improved even if the quadratic (or higher-order) finite element were employed.
	In fact, the domain perturbation term $I_4$ (see Lemmas \ref{lem: asymptotic Galerkin orthogonality} and \ref{lem: estimate of I4 in Sec3} below) gives rise to such contributions regardless of the choice of $V_h$, unless the approximation of $\Gamma$ becomes more accurate than $P_1$.
\end{rem}

Let us reduce pointwise error estimates to $W^{1,1}$-error analysis for regularized Green's functions, which is now a standard approach in this field.
For arbitrary $T \in \mathcal T_h$ and $x_0 \in T$ we let $\eta = \eta_{x_0} \in C_0^\infty(T)$, $\eta \ge 0$ be a regularized delta function such that
\begin{equation} \label{eq: properties of eta}
	\begin{aligned}
		&\int_T \eta(x) v_h(x) \, dx = v_h(x_0) \quad \forall v_h\in P_1(T), \qquad \|\nabla^k\eta\|_{L^\infty(T)} \le Ch^{-k} \, (k=0,1,2), \\
		&\operatorname{dist}(\operatorname{supp}\eta, \partial T) \ge Ch,
	\end{aligned}
\end{equation}
where $C$ is independent of $T, h$, and $x_0$ (see \cite{SSW96} for construction of $\eta$).
\begin{rem}
	(i) The quasi-uniformity of meshes are needed to ensure the last two properties of \eref{eq: properties of eta}. \\
	(ii) We have $\operatorname{supp}\eta \cap \Gamma(2\delta) = \emptyset$ with $\delta = C_{0E}h^2$, provided that $h$ is sufficiently small.
\end{rem}

We consider two kinds of regularized Green's functions $g_0, g_1 \in C^\infty(\overline\Omega)$ satisfying the following PDEs:
\begin{equation*}
	-\Delta g_0 + g_0 = \eta \quad\text{in}\quad\Omega, \qquad \partial_ng_0 = 0 \quad\text{on}\quad\Gamma,
\end{equation*}
and
\begin{equation*}
	-\Delta g_1 + g_1 = \partial\eta \quad\text{in}\quad\Omega, \qquad \partial_ng_1 = 0 \quad\text{on}\quad\Gamma,
\end{equation*}
where $\partial$ stands for an arbitrary directional derivative.
Accordingly, we let $g_{0h}, g_{1h} \in V_h$ be the solutions for finite element approximate problems as follows:
\begin{equation*}
	a_h(v_h, g_{0h}) = (v_h, \eta)_{\Omega_h} \qquad \forall v_h\in V_h, \qquad\text{and}\qquad a_h(v_h, g_{1h}) = (v_h, \partial\eta)_{\Omega_h} \qquad \forall v_h\in V_h.
\end{equation*}

The goal of this section is then to reduce \tref{main theorem} to the estimate
\begin{equation} \label{eq: W1,1 estimate for g-gh}
	\|\tilde g_m - g_{mh}\|_{W^{1,1}(\Omega_h)} \le C(h|\log h|)^{1-m}, \quad m=0,1,
\end{equation}
where $C$ is independent of $h, x_0$, and $\partial$, and $\tilde g_m := Pg_m$ means the extension defined in \sref{sec2.3}.
To observe this fact, we represent pointwise errors at $x_0$, with the help of $\eta$, as
\begin{align*}
	\tilde u(x_0) - u_h(x_0) &= (\tilde u - v_h)(x_0) + (v_h - \tilde u, \eta)_{\Omega_h} + (\tilde u - u_h, \eta)_{\Omega_h}, \\
	\partial(\tilde u - u_h)(x_0) &= \partial(\tilde u - v_h)(x_0) + (\partial(v_h - \tilde u), \eta)_{\Omega_h} - (\tilde u - u_h, \partial\eta)_{\Omega_h},
\end{align*}
for all $v_h \in V_h$.
Since the first two terms on the right-hand sides are bounded by $2\|\tilde u - v_h\|_{L^\infty(\Omega_h)}$ and $2\|\nabla(\tilde u - v_h)\|_{L^\infty(\Omega_h)}$, in order to prove \tref{main theorem} it suffices to show that
\begin{align*}
	|(\tilde u - u_h, \eta)_{\Omega_h}| &\le C h|\log h| \|\tilde u - v_h\|_{W^{1,\infty}(\Omega_h)} + Ch^2 |\log h| \|u\|_{W^{2,\infty}(\Omega)}, \\
	|(\tilde u - u_h, \partial\eta)_{\Omega_h}| &\le C\|\tilde u - v_h\|_{W^{1,\infty}(\Omega_h)} + Ch \|u\|_{W^{2,\infty}(\Omega)}.
\end{align*}
With this aim we prove:
\begin{prop} \label{prop: main result of Sec. 3.1}
	For $m=0,1$ and arbitrary $v_h \in V_h$, one obtains
	\begin{align*}
		|(\tilde u - u_h, \partial^m\eta)_{\Omega_h}| &\le C(\|\tilde u - v_h\|_{W^{1,\infty}(\Omega_h)} + Ch \|u\|_{W^{2,\infty}(\Omega)}) \|\tilde g_m - g_{mh}\|_{W^{1,1}(\Omega_h)} \\
			&\hspace{1cm} + Ch (h|\log h|)^{1-m} \|u\|_{W^{2,\infty}(\Omega)}.
	\end{align*}
\end{prop}

It is immediate to conclude \tref{main theorem} from \pref{prop: main result of Sec. 3.1} combined with \eref{eq: W1,1 estimate for g-gh}.
The rest of this section is thus devoted to the proof of \pref{prop: main result of Sec. 3.1}, whereas \eref{eq: W1,1 estimate for g-gh} will be established in Sections \ref{sec4}--\ref{sec6} below.
From now on, we suppress the subscript $m$ of $g_m$ and $g_{mh}$ for simplicity, as far as there is no fear of confusion.

Let us proceed to the proof of \pref{prop: main result of Sec. 3.1}.
Define functionals for $v \in H^1(\Omega_h)$, which will represent ``residuals'' of Galerkin orthogonality relation, by
\begin{align*}
	\operatorname{Res}_u(v) &= (-\Delta\tilde u + \tilde u - \tilde f, v)_{\Omega_h \setminus \Omega} + (\partial_{n_h} \tilde u - \tilde\tau, v)_{\Gamma_h}, \\
	\operatorname{Res}_g(v) &= (v, -\Delta\tilde g + \tilde g)_{\Omega_h \setminus \Omega} + (v, \partial_{n_h} \tilde g)_{\Gamma_h}.
\end{align*}
If in addition $v \in H^1(\tilde\Omega)$ in the expanded domain $\tilde\Omega = \Omega \cup \Gamma(\delta)$, then $\operatorname{Res}_u(v)$ admits another expression.
To observe this, we introduce ``signed'' integration defined as follows:
\begin{align*}
	(\phi, \psi)_{\Omega_h \triangle \Omega}' &:= (\phi, \psi)_{\Omega_h \setminus \Omega} - (\phi, \psi)_{\Omega \setminus \Omega_h}, \\
	(\phi, \psi)_{\Gamma_h \cup \Gamma}' &:= (\phi, \psi)_{\Gamma_h} - (\phi, \psi)_{\Gamma}, \\
	a_{\Omega_h \triangle \Omega}'(\phi, \psi) &:= (\nabla\phi, \nabla\psi)_{\Omega_h \triangle \Omega}' + (\phi, \psi)_{\Omega_h \triangle \Omega}'.
\end{align*}

\begin{lem} \label{lem: another representation of Res}
	For $v \in H^1(\tilde\Omega)$ we have
	\begin{equation*}
		\operatorname{Res}_u(v) = -(\tilde f, v)_{\Omega_h \triangle \Omega}' - (\tilde\tau, v)_{\Gamma_h \cup \Gamma}' + a_{\Omega_h \triangle \Omega}'(\tilde u, v).
	\end{equation*}	
\end{lem}
\begin{proof}
	Notice that the following integration-by-parts formula holds:
	\begin{equation*}
		(-\Delta \tilde u, v)_{\Omega_h \triangle \Omega}' = (\nabla\tilde u, \nabla v)_{\Omega_h \triangle \Omega}' - (\partial_{n_h}\tilde u, v)_{\Gamma_h} + (\partial_{n} u, v)_{\Gamma}.
	\end{equation*}
	From this formula and \eref{eq: Poisson} it follows that
	\begin{align*}
		(-\Delta \tilde u, v)_{\Omega_h \setminus \Omega} + (\partial_{n_h} \tilde u, v)_{\Gamma_h} &= (-\Delta u, v)_{\Omega \setminus \Omega_h} + (\nabla\tilde u, \nabla v)_{\Omega_h \triangle \Omega}' + (\partial_{n} u, v)_{\Gamma} \\
			&= -(u - f, v)_{\Omega \setminus \Omega_h} + (\nabla\tilde u, \nabla v)_{\Omega_h \triangle \Omega}' + (\tau, v)_{\Gamma}.
	\end{align*}
	Substituting this into the definition of $\operatorname{Res}_u(v)$ leads to the desired equality.
\end{proof}

Now we show that $\operatorname{Res}_u(\cdot)$ and $\operatorname{Res}_g(\cdot)$ represent residuals of Galerkin orthogonality relation for $\tilde u - u_h$ and $\tilde g - g_h$, respectively.
\begin{lem} \label{lem: asymptotic Galerkin orthogonality}
	For all $v_h \in V_h$ we have
	\begin{equation*}
		a_h(\tilde u - u_h, v_h) = \operatorname{Res}_u(v_h), \quad a_h(v_h, \tilde g - g_h) = \operatorname{Res}_g(v_h),
	\end{equation*}
	and
	\begin{align*}
		(\tilde u - u_h, \partial^{m}\eta)_{\Omega_h} &= a_h(\tilde u - v_h, \tilde g - g_h) - \operatorname{Res}_g(\tilde u - v_h) - \operatorname{Res}_u(\tilde g - g_h) + \operatorname{Res}_u(\tilde g) \\
			&=: I_1 + I_2 + I_3 + I_4.
	\end{align*}
\end{lem}
\begin{proof}
	From integration by parts and from the definitions of $u$ and $u_h$ we have
	\begin{align*}
		a_h(\tilde u - u_h, v_h) = (-\Delta\tilde u + \tilde u, v_h)_{\Omega_h} + (\partial_{n_h}\tilde u, v_h)_{\Gamma_h} - (\tilde f, v_h)_{\Omega_h} - (\tilde\tau, v_h)_{\Gamma_h} = \operatorname{Res}_u(v_h).
	\end{align*}
	The second equality is obtained in the same way.
	To show the third equality, we observe that
	\begin{align*}
		(v_h - u_h, \partial^m\eta)_{\Omega_h} &= a_h(v_h - u_h, g_h) = a_h(v_h - \tilde u, g_h) + a_h(\tilde u - u_h, g_h) \\
			&= a_h(\tilde u - v_h, \tilde g - g_h) - a_h(\tilde u - v_h, \tilde g) + \operatorname{Res}_u(g_h).
	\end{align*}
	It follows from integration by parts, $-\Delta g + g = \partial^m\eta$ in $\Omega$, and $\operatorname{supp}\eta\subset \Omega_h\cap\Omega$, that
	\begin{align*}
		a_h(\tilde u - v_h, \tilde g) &= (\tilde u - v_h, -\Delta\tilde g + \tilde g)_{\Omega_h} + (\tilde u - v_h, \partial_{n_h}\tilde g)_{\Gamma_h}
		= (u - v_h, \partial^m\eta)_{\Omega_h\cap\Omega} + \operatorname{Res}_g(\tilde u - v_h) \\
			&= (\tilde u - v_h, \partial^m\eta)_{\Omega_h} + \operatorname{Res}_g(\tilde u - v_h).
	\end{align*}
	Combining the two relations above yields the third equality.
\end{proof}

By the H\"older inequality, $|I_1| \le \|\tilde u - v_h\|_{W^{1,\infty}(\Omega_h)} \|\tilde g - g_h\|_{W^{1,1}(\Omega_h)}$.
The other terms are estimated in the following three lemmas.
There, boundary-skin estimates for $g$ will be frequently exploited, which are collected in the appendix.
\begin{lem}
	$|I_2| \le C(h|\log h|)^{1-m} \, \|\tilde u - v_h\|_{L^\infty(\Omega_h)}$.
\end{lem}
\begin{proof}
	By the H\"older inequality,
	\begin{equation*}
		|\operatorname{Res}_g(\tilde u - v_h)| \le \|\tilde u - v_h\|_{L^\infty(\Omega_h)} (\|\tilde g\|_{W^{2,1}(\Gamma(\delta))} + \|\partial_{n_h}\tilde g\|_{L^1(\Gamma_h)}),
	\end{equation*}
	where $\|\tilde g\|_{W^{2,1}(\Gamma(\delta))} \le Ch^{1-m}$ as a result of \cref{cor: boundary-skin estimates for g}.
	Since $(\nabla g)\circ\pi \cdot n\circ\pi = 0$ on $\Gamma_h$, it follows again from \cref{cor: boundary-skin estimates for g} that
	\begin{align*}
		\|\partial_{n_h}\tilde g\|_{L^1(\Gamma_h)} &\le \|\nabla\tilde g \cdot (n_h - n\circ\pi)\|_{L^1(\Gamma_h)} + \|\big( \nabla\tilde g - (\nabla\tilde g)\circ\pi \big) \cdot n\circ\pi\|_{L^1(\Gamma_h)} \\
			&\le Ch \|\nabla\tilde g\|_{L^1(\Gamma_h)} + C\|\nabla^2\tilde g\|_{L^1(\Gamma(\delta))} \le C (h|\log h|)^{1-m} + Ch^{1-m},
	\end{align*}
	which completes the proof.
\end{proof}

\begin{lem}
	$|I_3| \le Ch\|u\|_{W^{2,\infty}(\Omega)} \|\tilde g - g_h\|_{W^{1,1}(\Omega_h)}$.
\end{lem}
\begin{proof}
	By the H\"older inequality and stability of extensions,
	\begin{equation*}
		|\operatorname{Res}_u(\tilde g - g_h)| \le C\|u\|_{W^{2, \infty}(\Omega)} \|\tilde g - g_h\|_{L^1(\Omega_h\setminus\Omega)} + \|\partial_{n_h}\tilde u - \tilde\tau\|_{L^\infty(\Gamma_h)} \|\tilde g - g_h\|_{L^1(\Gamma_h)}.
	\end{equation*}
	From \eref{eq: boundary-skin estimate 2} and the trace theorem one has
	\begin{equation*}
		\|\tilde g - g_h\|_{L^1(\Omega_h\setminus\Omega)} \le C\delta (\|\tilde g - g_h\|_{L^1(\Gamma_h)} + \|\nabla(\tilde g - g_h)\|_{L^1(\Omega_h\setminus\Omega)})
		\le Ch^2 \|\tilde g - g_h\|_{W^{1,1}(\Omega_h)}.
	\end{equation*}
	From $(\nabla u)\circ\pi \cdot n\circ\pi = \tau\circ\pi$ on $\Gamma_h$, \eref{eq: boundary-skin estimates}, and the stability of extensions, it follows that
	\begin{align*}
		\|\partial_{n_h}\tilde u - \tilde\tau\|_{L^\infty(\Gamma_h)} &\le \|\nabla\tilde u\cdot (n_h - n\circ\pi)\|_{L^\infty(\Gamma_h)} + \|\big( \nabla\tilde u - (\nabla\tilde u)\circ\pi \big) \cdot n\circ\pi\|_{L^\infty(\Gamma_h)} + \|\tau\circ\pi - \tilde\tau\|_{L^\infty(\Gamma_h)} \\
			&\le Ch \|\nabla\tilde u\|_{L^\infty(\Gamma_h)} + C\delta\|\nabla^2\tilde u\|_{L^\infty(\Gamma(\delta))} + C\delta\|\nabla\tilde\tau\|_{L^\infty(\Gamma(\delta))}
			\le Ch \|u\|_{W^{2,\infty}(\Omega_h)}.
	\end{align*}
	Combining the estimates above and using the trace theorem once again, we conclude
	\begin{align*}
		|\operatorname{Res}_u(\tilde g - g_h)| &\le Ch^2 \|u\|_{W^{2,\infty}(\Omega)} \|\tilde g - g_h\|_{W^{1,1}(\Omega_h)} + Ch\|u\|_{W^{2,\infty}(\Omega_h)} \|\tilde g - g_h\|_{L^1(\Gamma_h)} \\
			&\le Ch\|u\|_{W^{2,\infty}(\Omega)} \|\tilde g - g_h\|_{W^{1,1}(\Omega_h)}.
	\end{align*}
	This completes the proof.
\end{proof}

\begin{lem} \label{lem: estimate of I4 in Sec3}
	$|I_4| \le Ch (h|\log h|)^{1-m} \|u\|_{W^{2,\infty}(\Omega)}$.
\end{lem}
\begin{proof}
	We recall from \lref{lem: another representation of Res} that
	\begin{equation*}
		\operatorname{Res}_u(\tilde g) = -(\tilde f, \tilde g)_{\Omega_h \triangle \Omega}' - (\tilde\tau, \tilde g)_{\Gamma_h \cup \Gamma}' + a_{\Omega_h \triangle \Omega}'(\tilde u, \tilde g).
	\end{equation*}
	Let us estimate each term in the right-hand side.
	By \eref{eq: boundary-skin estimates}$_2$ we obtain
	\begin{equation*}
		|(\tilde f, \tilde g)_{\Omega_h \triangle \Omega}'| \le \|\tilde f\|_{L^\infty(\Gamma(\delta))} \|\tilde g\|_{L^1(\Gamma(\delta))} \le C\delta |\log h|^{1-m} \|u\|_{W^{2,\infty}(\Omega)},
	\end{equation*}
	where $\delta= C_{0E}h^2$.  Next, from \eref{eq: boundary-skin estimates} and \cref{cor: boundary-skin estimates for g} we find that
	\begin{align*}
		(\tilde\tau, \tilde g)_{\Gamma_h \cup \Gamma}' &= |(\tau, g)_{\Gamma} - (\tilde\tau, \tilde g)_{\Gamma_h}| \le |(\tau, g)_\Gamma - (\tau\circ\pi, g\circ\pi)_{\Gamma_h}| + |(\tau\circ\pi, g\circ\pi - \tilde g)_{\Gamma_h}| + |(\tau\circ\pi - \tilde\tau, \tilde g)_{\Gamma_h}| \\
			&\le C\delta \|\tau\|_{L^\infty(\Gamma)} \|g\|_{L^1(\Gamma)} + C\|\tau\|_{L^\infty(\Gamma)} \|\nabla\tilde g\|_{L^1(\Gamma(\delta))} + C\delta \|\nabla\tilde\tau\|_{L^\infty(\Gamma(\delta))} \|\tilde g\|_{L^1(\Gamma_h)} \\
			&\le C\delta \|\nabla u\|_{L^\infty(\Omega)} |\log h|^m + C\|\nabla u\|_{L^\infty(\Omega)} \delta h^{-m} |\log h|^{1-m} + C\delta \|u\|_{W^{2,\infty}(\Omega)} |\log h|^{1-m} \\
			&\le C\delta h^{-m} |\log h|^{1-m} \|u\|_{W^{2,\infty}(\Omega)}.
	\end{align*}
	Finally, for the last term we obtain
	\begin{equation*}
		|a_{\Omega_h \triangle \Omega}'(\tilde u, \tilde g)| \le \|\tilde u\|_{W^{1,\infty}(\Gamma(\delta))} \|\tilde g\|_{W^{1,1}(\Gamma(\delta))} \le C\|u\|_{W^{1,\infty}(\Omega)} \delta h^{-m} |\log h|^{1-m}.
	\end{equation*}
	Collecting the above estimates proves the lemma.
\end{proof}

\pref{prop: main result of Sec. 3.1} is now an immediate consequence of Lemmas \ref{lem: asymptotic Galerkin orthogonality}--\ref{lem: estimate of I4 in Sec3}.

\section{Weighted $H^1$-estimates} \label{sec4}
As a consequence of the previous section, we need to estimate $\|\tilde g - g_h\|_{W^{1,1}(\Omega_h)}$, where we keep dropping the subscript $m$ (either 0 or 1) of $g_m$ and $g_{mh}$.
To this end we introduce a dyadic decomposition $\mathcal A_\Omega(x_0, d_0) = \{\Omega\cap A_j\}_{j=0}^J$ of $\Omega$, and observe from \eref{eq: L1 can be bdd by weighted Lp} that
\begin{equation} \label{eq: W1,1 is bdd by weighted H1}
	\|\tilde g - g_h\|_{W^{1,1}(\Omega_h)} \le C \sum_{j=0}^J d_j^{N/2} \|\tilde g - g_h\|_{H^1(\Omega_h \cap A_j)}.
\end{equation}
Then the weighted $H^1$-norm in the right-hand side is bounded as follows:

\begin{prop} \label{prop: weighted H1 estimate}
	There exists $K_0>0$ such that, for any dyadic decomposition $\mathcal A_\Omega(x_0, d_0) = \{\Omega \cap A_j\}_{j=0}^J$ of $\Omega$ with $d_0 = Kh$, $K \ge K_0$, we obtain
	\begin{equation} \label{eq: weighted H1 estimate}
		\sum_{j=0}^J d_j^{N/2} \|\tilde g - g_h\|_{H^1(\Omega_h \cap A_j)} \le CK^{m+N/2} h^{1-m} + C(h|\log h|)^{1-m} + C\sum_{j=0}^J d_j^{-1+N/2} \|\tilde g - g_h\|_{L^2(\Omega_h\cap A_j)}.
	\end{equation}
	Here the constants $K_0$ and $C$ are independent of $h, x_0, \partial$, and $K$.
\end{prop}

The rest of this section is devoted to the proof of the proposition above.
In order to estimate $\|\tilde g - g_h\|_{H^1(\Omega_h \cap A_j)}$ for $j = 0, \dots, J$, we use a cut off function $\omega_j \in C_0^\infty(\mathbb R^N), \; \omega_j\ge0$ such that
\begin{equation} \label{eq: properties of omega}
	\omega_j \equiv1 \quad\text{in}\quad A_j, \qquad \operatorname{supp}\omega_j \subset A_j^{(1/4)}, \qquad \|\nabla^k \omega_j\|_{L^\infty(\mathbb R^N)} \le Cd_j^{-k} \quad (k=0,1,2).
\end{equation}
Then we find that
\begin{align*}
	\|\tilde g - g_h\|_{H^1(\Omega_h \cap A_j)}^2 &\le \big( \omega_j (\tilde g - g_h), \tilde g - g_h \big)_{\Omega_h} + \big( \omega_j \nabla(\tilde g - g_h), \nabla(\tilde g - g_h) \big)_{\Omega_h} \\
		&= a_h\big( \omega_j(\tilde g - g_h), \tilde g - g_h \big) - \big( \nabla\omega_j(\tilde g - g_h), \nabla(\tilde g - g_h) \big)_{\Omega_h} \\
		&= a_h\big( \omega_j(\tilde g - g_h) - v_h, \tilde g - g_h \big) - \big( (\nabla\omega_j) (\tilde g - g_h), \nabla(\tilde g - g_h) \big)_{\Omega_h} + \operatorname{Res}_g(v_h) \\
		&=: I_1 + I_2 + I_3,
\end{align*}
where $v_h \in V_h$ is arbitrary and we have used \lref{lem: asymptotic Galerkin orthogonality}.

Substituting $v_h = \mathcal I_h(\omega_j(\tilde g - g_h))$, where $\mathcal I_h$ is the interpolation operator given in \sref{sec2.1}, we estimate $I_1, I_2$, and $I_3$ in the following.

\begin{lem}
	$I_1$ is bounded as
	\begin{align}
		|I_1| &\le Chd_j^{-2} \|\tilde g - g_h\|_{L^2(\Omega_h \cap A_j^{(1/2)})} \|\tilde g - g_h\|_{H^1(\Omega_h \cap A_j^{(1/2)})} + Chd_j^{-1} \|\tilde g - g_h\|_{H^1(\Omega_h \cap A_j^{(1/2)})}^2 \notag \\
			&\hspace{1cm}	+ C_j h d_j^{-m-N/2} \|\tilde g - g_h\|_{H^1(\Omega_h \cap A_j^{(1/2)})} \label{eq: estimate of I1 in Sec4},
	\end{align}
	where $C_0 = C K^{m+N/2}$ and $C_j = C$ for $1\le j\le J$.
\end{lem}
\begin{proof}
	By \cref{cor: support of interpolant} we have $\operatorname{supp}v_h \subset \Omega_h \cap A_j^{(1/2)}$, and hence
	\begin{align*}
		|I_1| \le \|\omega_j(\tilde g - g_h) - v_h\|_{H^1(\Omega_h)} \|\tilde g - g_h\|_{H^1(\Omega_h \cap A_j^{(1/2)})}.
	\end{align*}
	It follows from the interpolation error estimate, together with \eref{eq: properties of omega}, that
	\begin{align*}
		&\|\omega_j(\tilde g - g_h) - v_h\|_{H^1(\Omega_h)}^2 \le Ch^2 \sum_{T \in \mathcal T_h} \|\nabla^2 \big( \omega_j(\tilde g - g_h) \big)\|_{L^2(T)}^2 \\
		\le\; &Ch^2 \sum_{T\in\mathcal T_h} \Big( \|(\nabla^2\omega_j)(\tilde g - g_h)\|_{L^2(T)}^2 + \|(\nabla\omega_j) \otimes \nabla(\tilde g - g_h)\|_{L^2(T)}^2 + \|\nabla^2\tilde g\|_{L^2(T)}^2 \Big) \\
		\le\; &Ch^2 \sum_{T \cap A_j^{(1/4)} \neq \emptyset} \Big( d_j^{-4} \|\tilde g - g_h\|_{L^2(T)}^2 + d_j^{-2} \|\tilde g - g_h\|_{L^2(T)}^2 + \|\nabla^2\tilde g\|_{L^2(T)}^2 \Big),
	\end{align*}
	where we made use of $\nabla^2g_h|_{T} = 0$ for $T \in \mathcal T_h$.
	This combined with \lref{lem: consistency of annuli and triangulation}(i) implies
	\begin{equation*}
		\|\omega_j(\tilde g - g_h) - v_h\|_{H^1(\Omega_h)} \le Ch (d_j^{-2} \|\tilde g - g_h\|_{L^2(\Omega_h \cap A_j^{(1/2)})} + d_j^{-1} \|\tilde g - g_h\|_{L^2(\Omega_h \cap A_j^{(1/2)})} + \|\nabla^2\tilde g\|_{L^2(\Omega_h \cap A_j^{(1/2)})}).
	\end{equation*}
	When $j=0$, by the stability of extension and the $H^2$-regularity theory, we deduce that
	\begin{equation*}
		\|\nabla^2\tilde g\|_{L^2(\Omega_h\cap A_0^{(1/2)})} \le C\|g\|_{H^2(\Omega)} \le C\|\partial^m\eta\|_{L^2(\Omega)} \le Ch^{-m-N/2} = CK^{m+N/2} d_0^{-m-N/2}.
	\end{equation*}
	When $j \ge 1$, it follows from \lref{lem: local Lp estimate of g} that $\|\nabla^2\tilde g\|_{L^2(\Omega_h\cap A_0^{(1/2)})} \le Cd_j^{-m-N/2}$.
	Collecting the estimates above, we conclude \eref{eq: estimate of I1 in Sec4}.
\end{proof}

For $I_2$ we have
\begin{equation*}
	|I_2| \le Cd_j^{-1} \|\tilde g - g_h\|_{L^2(\Omega_h \cap A_j^{(1/2)})} \|\tilde g - g_h\|_{H^1(\Omega_h \cap A_j^{(1/2)})},
\end{equation*}
which dominates the first term in the right-hand side of \eref{eq: estimate of I1 in Sec4} because $hd_j^{-1} \le 1$.

\begin{lem}
	$|I_3| \le Ch d_j^{1/2-m-N/2} (\|\tilde g - g_h\|_{H^1(\Omega_h \cap A_j^{(1/4)})} + d_j^{-1} \|\tilde g - g_h\|_{L^2(\Omega_h \cap A_j^{(1/4)})})$.
\end{lem}
\begin{proof}
	Since $I_3 = (v_h, -\Delta\tilde g + \tilde g)_{\Omega_h\setminus\Omega} + (v_h, \partial_{n_h}\tilde g)_{\Gamma_h}$, we observe that
	\begin{align*}
		|(v_h, -\Delta\tilde g + \tilde g)_{\Omega_h\setminus\Omega}| &\le C\delta^{1/2} \|v_h\|_{H^1(\Omega_h)} (\delta d_j^{N-1})^{1/2} d_j^{-m-N} \le C\delta d_j^{-1/2-m-N/2} \|v_h\|_{H^1(\Omega_h)},
	\end{align*}
	and that
	\begin{align*}
		|(v_h, \partial_{n_h}\tilde g)_{\Gamma_h}| &\le \|v_h\|_{L^2(\Gamma_h)} \|\partial_{n_h} \tilde g\|_{L^2(\Gamma_h \cap A_j^{(1/4)})} \\
			&\le C\|v_h\|_{H^1(\Omega_h)} \big( \|\nabla\tilde g \cdot (n_h - n\circ\pi)\|_{L^2(\Gamma_h \cap A_j^{(1/4)})} + \|(\nabla\tilde g - (\nabla\tilde g)\circ\pi) \cdot n\circ\pi\|_{L^2(\Gamma_h \cap A_j^{(1/4)})} \big) \\
			&\le C\|v_h\|_{H^1(\Omega_h)} \big( h\|\nabla\tilde g\|_{L^2(\Gamma_h \cap A_j^{(1/4)})} + |\Gamma_h \cap A_j^{(1/4)}|^{1/2} \delta \|\nabla^2\tilde g\|_{L^\infty(\Gamma_h \cap A_j^{(1/4)})} \big) \\
			&\le C\|v_h\|_{H^1(\Omega_h)} (hd_j^{1/2 - m - N/2} + h^2d_j^{-1/2-m-N/2}) \le Ch d_j^{1/2-m-N/2} \|v_h\|_{H^1(\Omega_h)}.
	\end{align*}
	Therefore, by the $H^1$-stability of $\mathcal I_h$ and by $d_j \le 2\operatorname{diam}\Omega$,
	\begin{align*}
		|I_3| &\le Ch d_j^{1/2-m-N/2} \|\omega_j(\tilde g - g_h)\|_{H^1(\Omega_h)} \\
			&\le Ch d_j^{1/2-m-N/2} (\|\tilde g - g_h\|_{H^1(\Omega_h \cap A_j^{(1/4)})} + d_j^{-1} \|\tilde g - g_h\|_{L^2(\Omega_h \cap A_j^{(1/4)})}),
	\end{align*}
	which completes the proof.
\end{proof}

Collecting the estimates for $I_1, I_2$, and $I_3$ we deduce that
\begin{align*}
	d_j^{N/2} \|\tilde g - g_h\|_{H^1(\Omega_h \cap A_j)} &\le C(hd_j^{-1})^{1/2} d_j^{N/2} \|\tilde g - g_h\|_{H^1(\Omega_h \cap A_j^{(1)})} \\
	&\quad + C (d_j^{N/2}\|\tilde g - g_h\|_{H^1(\Omega_h \cap A_j^{(1)})})^{1/2} ( d_j^{-1+N/2}\|\tilde g - g_h\|_{L^2(\Omega_h \cap A_j^{(1)})} )^{1/2} \\
	&\quad + \big( C_j hd_j^{-m} (1 + d_j^{1/2}) \big)^{1/2} (d_j^{N/2}\|\tilde g - g_h\|_{H^1(\Omega_h \cap A_j^{(1)})})^{1/2} \\
	&\quad + C (hd_j^{1/2-m})^{1/2} \big( d_j^{-1+N/2}\|\tilde g - g_h\|_{L^2(\Omega_h \cap A_j^{(1)})})^{1/2}.
\end{align*}
We now take the summation for $j = 0, 1, \dots, J$ and apply \eref{eq: reduction of expanded annuli to usual ones} to have
\begin{align*}
	\sum_{j=0}^J d_j^{N/2} \|\tilde g - g_h\|_{H^1(\Omega_h \cap A_j)} &\le C'(hd_0^{-1})^{1/2} \sum_{j=0}^J d_j^{N/2} \|\tilde g - g_h\|_{H^1(\Omega_h\cap A_j)} + \frac14 \sum_{j=0}^J d_j^{N/2} \|\tilde g - g_h\|_{H^1(\Omega_h\cap A_j)} \\
		&\hspace{-5mm} + \sum_{j=0}^J C_jh d_j^{-m}(1 + d_j^{1/2}) + Ch\sum_{j=0}^J d_j^{1/2-m} + C \sum_{j=0}^J d_j^{-1+N/2} \|\tilde g - g_h\|_{L^2(\Omega_h \cap A_j)}.
\end{align*}
If $hd_0^{-1} = K^{-1} \le 1/(4C')^2$, then one can absorb the first two terms into the left-hand side to conclude \eref{eq: weighted H1 estimate}.
This completes the proof of \pref{prop: weighted H1 estimate}.

Thus we are left to deal with $\sum_{j=0}^J d_j^{-1+N/2} \|\tilde g - g_h\|_{L^2(\Omega_h \cap A_j)}$, which will be the scope of the next section.

\section{Weighted $L^2$-estimates} \label{sec5}
Let us give estimation of the weighted $L^2$-norm appearing in the last term of \eref{eq: weighted H1 estimate}.
\begin{prop} \label{prop: weighted L2 estimate}
	There exists $K_0>0$ such that, for any dyadic decomposition $\mathcal A_\Omega(x_0, d_0) = \{\Omega \cap A_j\}_{j=0}^J$ of $\Omega$ with $d_0 = Kh$, $K_0 \le K \le h^{-1}$, we obtain
	\begin{align}
		&\sum_{j=0}^J d_j^{-1+N/2} \|\tilde g - g_h\|_{L^2(\Omega_h \cap A_j)} \notag \\
		\le\; & C(hd_0^{-1}) \Big( \sum_{j=0}^J d_j^{N/2} \|\tilde g - g_h\|_{H^1(\Omega_h\cap A_j)} + \|\tilde g - g_h\|_{W^{1,1}(\Omega_h)} \Big) + Ch^{3/2 - m}, \label{eq: weighted L2 estimate}
	\end{align}
	where the constants $K_0$ and $C$ are independent of $h, x_0, \partial$, and $K$.
\end{prop}
To prove this, first we fix $j=0, \dots, J$ and estimate $\|\tilde g - g_h\|_{L^2(\Omega_h \cap A_j)}$ based on a localized version of the Aubin--Nitsche trick.
In fact, since
\begin{equation*}
	\|\tilde g - g_h\|_{L^2(\Omega_h \cap A_j)} = \sup_{\substack{\varphi \in C_0^\infty(\Omega_h \cap A_j) \\ \|\varphi\|_{L^2(\Omega_h \cap A_j)} = 1}} (\varphi, \tilde g - g_h)_{\Omega_h},
\end{equation*}
it suffices to examine $(\varphi, \tilde g - g_h)_{\Omega_h}$ for such $\varphi$.
To express this quantity with a solution of a dual problem, we consider
\begin{equation} \label{eq: PDE for w}
	-\Delta w + w = \varphi \quad\text{in}\quad \Omega, \qquad \partial_n w = 0 \quad\text{on}\quad \Gamma,
\end{equation}
where $\varphi$ is extended by 0 to the outside of $\Omega_h \cap A_j$.
From the elliptic regularity theory we know that the solution $w$ is smooth enough.
We then obtain the following:
\begin{lem}
	For all $w_h \in V_h$ we have
	\begin{align}
		(\varphi, \tilde g - g_h)_{\Omega_h} &= a_h(\tilde w - w_h, \tilde g - g_h) - \operatorname{Res}_w(\tilde g - g_h) - \operatorname{Res}_g(\tilde w - w_h) + \operatorname{Res}_g(\tilde w) \notag \\
			&=: I_1 + I_2 + I_3 + I_4, \label{eq: representation of L2 error of g-gh}
	\end{align}
	where $\tilde w := Pw$ and $\operatorname{Res}_w : H^1(\Omega_h) \to \mathbb R$ is given by
	\begin{equation*}
		\operatorname{Res}_w(v) := (-\Delta\tilde w + \tilde w - \varphi, v)_{\Omega_h\setminus\Omega} + (\partial_{n_h}\tilde w, v)_{\Gamma_h}.
	\end{equation*}
\end{lem}
\begin{proof}
	We see that
	\begin{align*}
		(\varphi, \tilde g - g_h)_{\Omega_h} &= (\varphi, g - g_h)_{\Omega_h\cap\Omega} + (\varphi, \tilde g - g_h)_{\Omega_h\setminus\Omega} \\
			&= (-\Delta\tilde w + \tilde w, \tilde g - g_h)_{\Omega_h} + (\Delta\tilde w - \tilde w + \varphi, \tilde g - g_h)_{\Omega_h\setminus\Omega} \\
			&= a_h(\tilde w, \tilde g - g_h) - (\partial_{n_h}\tilde w, \tilde g - g_h)_{\Gamma_h} + (\Delta\tilde w - \tilde w + \varphi, \tilde g - g_h)_{\Omega_h\setminus\Omega} \\
			&= a_h(\tilde w - w_h, \tilde g - g_h) + \operatorname{Res}_g(w_h) - \operatorname{Res}_w(g - g_h),
	\end{align*}
	where we have used $a_h(w_h, \tilde g - g_h) = \operatorname{Res}_g(w_h)$ from \lref{lem: asymptotic Galerkin orthogonality}.
	This yields the desired equality.
\end{proof}
\begin{rem}
	In a similar way to \lref{lem: another representation of Res}, one can derive another expression for $\operatorname{Res}_g(v)$ if $v \in H^1(\tilde\Omega)$:
	\begin{equation*}
		\operatorname{Res}_g(v) = a_{\Omega_h\triangle\Omega}'(v, \tilde g).
	\end{equation*}
\end{rem}

In the following four lemmas, taking $w_h = \mathcal I_h\tilde w$, we estimate $I_1, I_2, I_3$, and $I_4$ by dividing the integrals over $\Omega_h$, $\Gamma_h$, or boundary-skin layers, into those defined near $A_j$ and away from $A_j$.
The former will be bounded, e.g., by the H\"older inequality of the form $\|\phi\|_{L^2(\Omega_h)} \|\psi\|_{L^2(\Omega_h \cap A_j^{(1/2)})}$ together with $H^2$-regularity estimates for $w$, whereas the latter will be bounded by $\|\phi\|_{L^\infty(\Omega_h \setminus A_j^{(1/2)})} \|\psi\|_{L^1(\Omega_h)}$ together with Green's function estimates for $w$ (see \lref{lem: estimate for w}).

\begin{lem} \label{lem: estimate of I1 in Sec5}
	$|I_1| \le Ch \|\tilde g - g_h\|_{H^1(\Omega_h \cap A_j^{(1/2)})} + Chd_j^{-N/2} \|\tilde g - g_h\|_{W^{1,1}(\Omega_h)}.$
\end{lem}
\begin{proof}
	By the H\"older inequality mentioned above,
	\begin{align*}
		|I_1| \le \|\tilde w - w_h\|_{H^1(\Omega_h)} \|\tilde g - g_h\|_{H^1(\Omega_h\cap A_j^{(1/2)})} + \|\tilde w - w_h\|_{W^{1,\infty}(\Omega_h\setminus A_j^{(1/2)})} \|\tilde g - g_h\|_{W^{1,1}(\Omega_h)},
	\end{align*}
	where we notice that
	\begin{equation*}
		\|\tilde w - w_h\|_{H^1(\Omega_h)} \le Ch \|w\|_{H^2(\Omega)} \le Ch \|\varphi\|_{L^2(\mathbb R^N)} = Ch,
	\end{equation*}
	and from \lref{lem: estimate for w} that
	\begin{equation*}
		\|\tilde w - w_h\|_{W^{1,\infty}(\Omega_h\setminus A_j^{(1/2)})} \le Ch \|\nabla^2\tilde w\|_{L^\infty(\Omega_h\setminus A_j^{(1/4)})} \le Chd_j^{-N/2}.
	\end{equation*}
	This completes the proof.
\end{proof}

\begin{lem}
	$I_2$ is bounded as
	\begin{equation*}
		|I_2| \le Ch^{1/2} \|\tilde g - g_h\|_{L^2(\Omega_h \cap A_j^{(3/4)})} + Ch \|\tilde g - g_h\|_{H^1(\Omega_h \cap A_j^{(3/4)})} + Chd_j^{-N/2} \|\tilde g - g_h\|_{W^{1,1}(\Omega_h)}.
	\end{equation*}
\end{lem}
\begin{proof}
	Recall that $I_2 = (\Delta\tilde w - \tilde w + \varphi, \tilde g - g_h)_{\Omega_h\setminus\Omega} - (\partial_{n_h}\tilde w, \tilde g - g_h)_{\Gamma_h} =: I_{21} + I_{22}$.
	Noting that $\varphi = 0$ in $\Omega_h\setminus A_j^{(1/2)}$ we estimate $I_{21}$ by
	\begin{align*}
		|I_{21}| &\le C(\|w\|_{H^2(\Omega)} + \|\varphi\|_{L^2(\mathbb R^N)}) \|\tilde g - g_h\|_{L^2( (\Omega_h\setminus\Omega) \cap A_j^{(1/2)})} + \|\tilde w\|_{W^{2,\infty}(\Omega_h\setminus A_j^{(1/2)})} \|\tilde g - g_h\|_{L^1(\Omega_h\setminus\Omega)} \\
			&\le C\|\tilde g - g_h\|_{L^2( (\Omega_h\setminus\Omega) \cap A_j^{(1/2)})} + Cd_j^{-N/2} \|\tilde g - g_h\|_{L^1(\Omega_h\setminus\Omega)}.
	\end{align*}
	To address the first term we introduce $\omega_j' \in C_0^\infty(\mathbb R^N), \, \omega_j' \ge 0$ such that
	\begin{equation*}
		\omega_j' \equiv 1 \quad\text{in}\quad A_j^{(1/2)}, \qquad \operatorname{supp}\omega_j' \subset A_j^{(3/4)}, \qquad \|\nabla^k \omega_j'\|_{L^\infty(\mathbb R^N)} \le Cd_j^{-k} \; (k=0,1,2).
	\end{equation*}
	Then it follows from \eref{eq: boundary-skin estimate 2} and the trace estimate that
	\begin{align}
		\|\tilde g - g_h\|_{L^2( (\Omega_h\setminus\Omega) \cap A_j^{(1/2)})} &\le \|\omega_j'(\tilde g - g_h)\|_{L^2(\Omega_h\setminus\Omega)} \notag \\
			&\le C\delta^{1/2} \|\omega_j' (\tilde g - g_h)\|_{L^2(\Gamma_h)} + C\delta \|\nabla \big( \omega_j' (\tilde g - g_h) \big)\|_{L^2(\Omega_h\setminus\Omega)} \notag \\
			&\le Ch \|\omega_j' (\tilde g - g_h)\|_{L^2(\Omega_h)}^{1/2} \|\omega_j' (\tilde g - g_h)\|_{H^1(\Omega_h)}^{1/2} \notag \\
			&\hspace{1cm} + Ch^2d_j^{-1} \|\tilde g - g_h\|_{L^2(\Omega_h \cap A_j^{(3/4)})} + Ch^2 \|\nabla (\tilde g - g_h)\|_{L^2(\Omega_h \cap A_j^{(3/4)})} \notag \\
			&\le Ch(1 + d_j^{-1/2}) \|\tilde g - g_h\|_{L^2(\Omega_h \cap A_j^{(3/4)})} + Ch \|\tilde g - g_h\|_{H^1(\Omega_h \cap A_j^{(3/4)})} \notag \\
			&\le Ch^{1/2} \|\tilde g - g_h\|_{L^2(\Omega_h \cap A_j^{(3/4)})} + Ch \|\tilde g - g_h\|_{H^1(\Omega_h \cap A_j^{(3/4)})}, \label{eq1: lem5.3}
	\end{align}
	where we have used $hd_j^{-1} \le 1$ and $h\le 1$. Again by \eref{eq: boundary-skin estimate 2} we also have
	\begin{align*}
		\|\tilde g - g_h\|_{L^1(\Omega_h\setminus\Omega)} \le C\delta (\|\tilde g - g_h\|_{L^1(\Gamma_h)} + \|\nabla(\tilde g - g_h)\|_{L^1(\Omega_h\setminus\Omega)}) \le Ch^2\|\tilde g - g_h\|_{W^{1,1}(\Omega_h)}.
	\end{align*}
	Combining the estimates above now gives
	\begin{equation} \label{eq: estimate of I21 in Sec5}
		|I_{21}| \le Ch^{1/2} \|\tilde g - g_h\|_{L^2(\Omega_h \cap A_j^{(3/4)})} + Ch \|\tilde g - g_h\|_{H^1(\Omega_h \cap A_j^{(3/4)})} + Ch^2d_j^{-N/2} \|\tilde g - g_h\|_{W^{1,1}(\Omega_h)}.
	\end{equation}
	
	Next we estimate $I_{22}$ by
	\begin{equation*}
		|I_{22}| \le \|\partial_{n_h}\tilde w\|_{L^2(\Gamma_h)} \|\tilde g - g_h\|_{L^2(\Gamma_h\cap A_j^{(1/2)})} + \|\partial_{n_h}\tilde w\|_{L^\infty(\Gamma_h \setminus A_j^{(1/2)})} \|\tilde g - g_h\|_{L^1(\Gamma_h)}.
	\end{equation*}
	For the first term we see that
	\begin{align*}
		\|\partial_{n_h}\tilde w\|_{L^2(\Gamma_h)} &\le \|\nabla\tilde w \cdot (n_h - n\circ\pi)\|_{L^2(\Gamma_h)} + \|\big( \nabla\tilde w - (\nabla\tilde w)\circ\pi \big)\cdot n\circ\pi\|_{L^2(\Gamma_h)} \\
			&\le Ch \|\nabla\tilde w\|_{L^2(\Gamma_h)} + C\delta^{1/2} \|\nabla^2\tilde w\|_{L^2(\Gamma(\delta))} \le Ch\|w\|_{H^2(\Omega)} \le Ch,
	\end{align*}
	and, in a similar way as we derived \eref{eq1: lem5.3}, that
	\begin{equation*}
		\|\tilde g - g_h\|_{L^2(\Gamma_h\cap A_j^{(1/2)})} \le Cd_j^{-1/2} \|\tilde g - g_h\|_{L^2(\Omega_h \cap A_j^{(3/4)})} + C \|\tilde g - g_h\|_{H^1(\Omega_h \cap A_j^{(3/4)})}.
	\end{equation*}
	For the second term, observe that
	\begin{align*}
		\|\partial_{n_h}\tilde w\|_{L^\infty(\Gamma_h \setminus A_j^{(1/2)})} &\le \|\nabla\tilde w \cdot (n_h - n\circ\pi)\|_{L^\infty(\Gamma_h \setminus A_j^{(1/2)})} + \|\big( \nabla\tilde w - (\nabla\tilde w)\circ\pi \big)\cdot n\circ\pi\|_{L^\infty(\Gamma_h \setminus A_j^{(1/2)})} \\
			&\le Ch \|\nabla\tilde w\|_{L^\infty(\Gamma(\delta) \setminus A_j^{(1/2)} )} + C\delta \|\nabla^2\tilde w\|_{L^\infty(\Gamma(\delta) \setminus A_j^{(1/4)})} \\
			&\le Chd_j^{1-N/2} + Ch^2d_j^{-N/2} \le Chd_j^{1-N/2},
	\end{align*}
	and that $\|\tilde g - g_h\|_{L^1(\Gamma_h)} \le C\|\tilde g - g_h\|_{W^{1,1}(\Omega_h)}$.
	Combining these estimates, we deduce
	\begin{equation} \label{eq: estimate of I22 in Sec5}
		|I_{22}| \le Chd_j^{-1/2} \|\tilde g - g_h\|_{L^2(\Omega_h \cap A_j^{(3/4)})} + Ch \|\tilde g - g_h\|_{H^1(\Omega_h \cap A_j^{(3/4)})} + Chd_j^{1-N/2} \|\tilde g - g_h\|_{W^{1,1}(\Omega_h)}.
	\end{equation}
	
	From \eref{eq: estimate of I21 in Sec5} and \eref{eq: estimate of I22 in Sec5}, together with $h\le d_j\le 2\operatorname{diam}\Omega$, we conclude the desired estimate.
\end{proof}

\begin{lem}
	$|I_3| \le Ch^{5/2-m} d_j^{-N/2}$.
\end{lem}
\begin{proof}
	Recall that $I_3 = (\tilde w - w_h, \Delta\tilde g - \tilde g)_{\Omega_h\setminus\Omega} - (\tilde w - w_h, \partial_{n_h}\tilde g)_{\Gamma_h} =: I_{31} + I_{32}$.
	We estimate $I_{31}$ by
	\begin{align*}
		|I_{31}| &\le \|\tilde w - w_h\|_{L^2(\Omega_h)} \|\tilde g\|_{H^2(\Gamma(\delta) \cap A_j^{(1/2)})} + \|\tilde w - w_h\|_{L^\infty(\Omega_h \setminus A_j^{(1/2)})} \|\tilde g\|_{W^{2,1}(\Gamma(\delta))} \\
			&\le Ch^2 \|\nabla^2\tilde w\|_{L^2(\Omega_h)} (\delta d_j^{N-1})^{1/2} d_j^{-m-N} + Ch^2 \|\nabla^2\tilde w\|_{L^\infty(\Omega_h \setminus A_j^{1/4})} \delta d_0^{-1-m} \\
			&\le Ch^3 d_j^{-1/2-m-N/2} + Ch^2d_j^{-N/2} h^{1-m} \le Ch^{5/2-m} d_j^{-N/2},
	\end{align*}
	where we have used $h\le d_j$.
	
	It remains to consider $I_{32}$; we estimate it by
	\begin{equation*}
		\|\tilde w - w_h\|_{L^2(\Gamma_h)} \|\partial_{n_h}\tilde g\|_{L^2(\Gamma_h \cap A_j^{(1/2)})} + \|\tilde w - w_h\|_{L^\infty(\Gamma_h \setminus A_j^{(1/2)})} \|\partial_{n_h}\tilde g\|_{L^1(\Gamma_h)}.
	\end{equation*}
	For the first term, we have $\|\tilde w - w_h\|_{L^2(\Gamma_h)} \le Ch^{3/2} \|\nabla^2\tilde w\|_{L^2(\Omega_h)} \le Ch^{3/2}$ and
	\begin{align*}
		&\|\partial_{n_h}\tilde g\|_{L^2(\Gamma_h \cap A_j^{(1/2)})} \\
		\le\;& |\Gamma_h \cap A_j^{(1/2)}|^{1/2} \big( \|\nabla\tilde g \cdot (n_h - n\circ\pi)\|_{L^\infty(\Gamma_h \cap A_j^{(1/2)})} + \|\nabla\tilde g - (\nabla\tilde g)\circ\pi\|_{L^\infty(\Gamma_h \cap A_j^{(1/2)})} \big) \\
		\le\;& Cd_j^{(N-1)/2} (h\|\nabla\tilde g\|_{L^\infty(\Gamma_h \cap A_j^{(1/2)})} + \delta \|\nabla^2\tilde g\|_{L^\infty(\Gamma(\delta) \cap A_j^{(3/4)})}) \\
		\le\;& Cd_j^{(N-1)/2} (hd_j^{1-m-N} + h^2 d_j^{-m-N}) \le Chd_j^{1/2-m-N/2}.
	\end{align*}
	For the second term, we have $\|\tilde w - w_h\|_{L^\infty(\Gamma_h \setminus A_j^{(1/2)})} \le Ch^2\|\nabla^2\tilde w\|_{L^\infty(\Gamma_h \setminus A_j^{(1/4)})} \le Ch^2d_j^{-N/2}$ and we find from \cref{cor: boundary-skin estimates for g} that $\|\partial_{n_h}\tilde g\|_{L^1(\Gamma_h)} \le C(h |\log h|)^{1-m} \le Ch^{(1-m)/2}$.
	Therefore,
	\begin{equation*}
		|I_{32}| \le C h^{5/2} d_j^{1/2-m-N/2} + Ch^{5/2-m/2} d_j^{-N/2} \le Ch^{5/2-m} d_j^{-N/2},
	\end{equation*}
	which completes the proof.
\end{proof}

\begin{lem} \label{lem: estimate of I4 in Sec5}
	$|I_4| \le Ch^2 d_j^{1/2-m-N/2} + Ch^{2-m} |\log h|^{1-m} d_j^{1-N/2}$.
\end{lem}
\begin{proof}
	We estimate $I_4 = a_{\Omega_h \triangle \Omega}'(\tilde w, \tilde g)$ by
	\begin{align*}
		|I_4| \le \|\tilde w\|_{H^1(\Gamma(\delta))} \|\tilde g\|_{H^1(\Gamma(\delta) \cap A_j^{(1/2)})} + \|\tilde w\|_{W^{1,\infty}(\Gamma(\delta) \setminus A_j^{(1/2)})} \|\tilde g\|_{W^{1,1}(\Gamma(\delta))}.
	\end{align*}
	The first term of the right-hand side is bounded, using \eref{eq: boundary-skin estimates}$_2$ and \lref{lem: local Lp estimate of g in boundary-skins}, by
	\begin{equation*}
		C\delta^{1/2}\|w\|_{H^2(\Omega)} (\delta d_j^{N-1})^{1/2} d_j^{1-m-N} \le Ch^2 d_j^{1/2-m-N/2}.
	\end{equation*}
	The second term is bounded, in view of \lref{lem: estimate for w} and \cref{cor: boundary-skin estimates for g}, by $Cd_j^{1-N/2} \delta h^{-m} |\log h|^{1-m}$.
	This completes the proof.
\end{proof}

Now we substitute the results of Lemmas \ref{lem: estimate of I1 in Sec5}--\ref{lem: estimate of I4 in Sec5} into \eref{eq: representation of L2 error of g-gh} and multiply by $d_j^{-1+N/2}$ to obtain
\begin{align}
	&d_j^{-1+N/2} \|\tilde g - g_h\|_{L^2(\Omega_h \cap A_j)} \notag \\
	\le\; & C(hd_j^{-1}) d_j^{N/2} \|\tilde g - g_h\|_{H^1(\Omega_h \cap A_j^{(1)})} + C(hd_j^{-1}) \|\tilde g - g_h\|_{W^{1,1}(\Omega_h)} \notag \\
	&\hspace{5mm} + Ch^{1/2} d_j^{-1+N/2} \|\tilde g - g_h\|_{L^2(\Omega_h \cap A_j^{(1)})} + Ch^{5/2-m} d_j^{-1} + Ch^2 d_j^{-1/2-m} + Ch^{2-m} |\log h|^{1-m}. \label{eq1: Sec5}
\end{align}
Taking the summation for $j=0,\dots,J$, assuming $h$ is sufficiently small and using \eref{eq: reduction of expanded annuli to usual ones}, we are able to absorb the third term in the right-hand side of \eref{eq1: Sec5} and then arrive at
\begin{align*}
	\sum_{j=0}^J d_j^{-1+N/2} \|\tilde g - g_h\|_{L^2(\Omega_h \cap A_j)} &\le C(hd_0^{-1}) \Big( \sum_{j=0}^J d_j^{N/2} \|\tilde g - g_h\|_{H^1(\Omega_h\cap A_j)} + \|\tilde g - g_h\|_{W^{1,1}(\Omega_h)} \Big) \\
		&\qquad + Ch^{5/2-m} d_0^{-1} + Ch^2 d_0^{-1/2-m} + Ch^{2-m} |\log h|^{1-m} |\log d_0|,
\end{align*}
where we note that the last three terms can be estimated by $Ch^{3/2-m}$ because $d_0 = Kh\le 1$ and $K>1$.
This completes the proof of \pref{prop: weighted L2 estimate}.

\section{End of the proof of the main theorem} \label{sec6}
Substituting \eref{eq: weighted L2 estimate} into \eref{eq: weighted H1 estimate} we obtain
\begin{align*}
	\sum_{j=0}^J d_j^{N/2} \|\tilde g - g_h\|_{H^1(\Omega_h \cap A_j)} &\le C''K^{-1} \Big( \sum_{j=0}^J d_j^{N/2} \|\tilde g - g_h\|_{H^1(\Omega_h\cap A_j)} + \|\tilde g - g_h\|_{W^{1,1}(\Omega_h)} \Big) \\
	&\hspace{1cm} + CK^{m+N/2}h^{1-m} + C(h|\log h|)^{1-m}.
\end{align*}
If $K \ge 2C''$, then it follows that
\begin{equation*}
	\sum_{j=0}^J d_j^{N/2} \|\tilde g - g_h\|_{H^1(\Omega_h \cap A_j)} \le CK^{-1} \|\tilde g - g_h\|_{W^{1,1}(\Omega_h)} + CK^{m+N/2}h^{1-m} + C(h|\log h|)^{1-m},
\end{equation*}
which combined with \eref{eq: W1,1 is bdd by weighted H1} yields
\begin{equation*}
	\|\tilde g - g_h\|_{W^{1,1}(\Omega_h)} \le C'''K^{-1} \|\tilde g - g_h\|_{W^{1,1}(\Omega_h)} + CK^{m+N/2}h^{1-m} + C(h|\log h|)^{1-m}.
\end{equation*}
If $K \ge 2C'''$, then this implies the desired estimate \eref{eq: W1,1 estimate for g-gh}, which together with \pref{prop: main result of Sec. 3.1} completes the proof of \tref{main theorem}.

\section{Numerical example} \label{sec7}
Letting $\Omega = \{(x, y) \in \mathbb R^2 \,:\, \frac{(x-0.12)^2}4 + \frac{(y+0.2)^2}9 < 1, \; (x-0.7)^2 + (y-0.1)^2 > 0.5^2\}$, which is non-convex, we set an exact solution to be $u(x, y) = x^2$.
We define $f$ and $\tau$ so that \eref{eq: Poisson} holds.
They have natural extensions to $\mathbb R^2$, which are exploited as $\tilde f$ and $\tilde \tau$.
Then we compute approximate solutions $u_h^k$ of \eref{eq: FEM} based on the $P_k$-finite elements ($k=1,2,3$), using the software \verb|FreeFEM++| \cite{Hec12}.
The errors $\|u - u_h^k\|_{L^\infty(\Omega_h)}$ and $\|\nabla(u - u_h^k)\|_{L^\infty(\Omega_h)}$, which are calculated with the use of $P_4$-finite element spaces, are reported in Tables \ref{tab1} and \ref{tab2}, respectively.

We see that the result for $k = 1$ is in accordance with \tref{main theorem}.
The one for $k = 3$ (although it is not covered by our theory) is also consistent with our theoretical expectation made in \rref{rem1 in sec3}(iii).
When $k = 2$, the $L^\infty$-error remains sub-optimal convergence as expected.
However, the $W^{1,\infty}$-error seems to be $O(h^2)$, which is significantly better than in the $P_3$-case.
We remark that such behavior was also observed for different (and apparently more complicated) choices of $\Omega$ and $u$.
There might be a super-convergence phenomenon in the $P_2$-approximation for Neumann problems in 2D smooth domains.

\begin{rem}
	If $k \ge 2$ and $\tilde\tau$ is chosen as $\nabla u\cdot n_h$, then $u_h^k$ agrees with $u$ (note that the above $u$ is quadratic), because this amounts to assuming that the original problem \eref{eq: Poisson} is given in a polygon $\Omega_h$.
	This was observed in our numerical experiment as well (up to rounding errors).
	However, since such $\tilde\tau$ is unavailable without knowing an exact solution, one cannot expect it in a practical computation.
\end{rem}

\begin{table}[htbp]
	\centering
	\caption{Behavior of the $L^\infty$-errors for the $P_k$-approximation ($k=1,2,3$)}
	\label{tab1}
	\begin{tabular}{ccccccc}
		$h$ & $\|u - u_h^1\|_{L^\infty(\Omega_h)}$ & rate & $\|u - u_h^2\|_{L^\infty(\Omega_h)}$ & rate & $\|u - u_h^3\|_{L^\infty(\Omega_h)}$ & rate \\
		\hline
		0.617 & 5.72e-2 & --- & 1.89e-2 & --- & 2.08e-2 & --- \\
		0.314 & 1.75e-2 & 1.8 & 4.39e-3 & 2.2 & 5.07e-3 & 2.1 \\
		0.165 & 4.64e-3 & 2.1 & 1.05e-3 & 2.2 & 1.30e-3 & 2.1 \\
		0.085 & 1.42e-3 & 1.8 & 2.55e-4 & 2.1 & 3.33e-4 & 2.1 \\
		0.043 & 3.92e-4 & 1.9 & 6.28e-5 & 2.1 & 8.31e-5 & 2.1
	\end{tabular}
\end{table}
\begin{table}[htbp]
	\centering
	\caption{Behavior of the $W^{1,\infty}$-errors for the $P_k$-approximation ($k=1,2,3$)}
	\label{tab2}
	\begin{tabular}{ccccccc}
		$h$ & $\|\nabla(u - u_h^1)\|_{L^\infty(\Omega_h)}$ & rate & $\|\nabla(u - u_h^2)\|_{L^\infty(\Omega_h)}$ & rate & $\|\nabla(u - u_h^3)\|_{L^\infty(\Omega_h)}$ & rate \\
		\hline
		0.617 & 6.24e-1 & --- & 9.98e-2 & --- & 3.91e-1 & --- \\
		0.314 & 3.21e-1 & 1.0 & 2.68e-2 & 1.9 & 2.15e-1 & 0.9 \\
		0.165 & 1.58e-1 & 1.1 & 6.85e-3 & 2.1 & 1.04e-1 & 1.1 \\
		0.085 & 9.18e-2 & 0.8 & 1.58e-3 & 2.2 & 5.47e-2 & 1.0 \\
		0.043 & 4.63e-2 & 1.0 & 4.42e-4 & 1.9 & 2.77e-2 & 1.0
	\end{tabular}
\end{table}

\appendix
\section{Auxiliary boundary-skin estimates}
\subsection{Local coordinate representation}
We exploit the notations and observations given in \cite[Section 8]{KOZ16}, which we briefly describe here.
Since $\Omega$ is a bounded $C^\infty$-domain, there exist a system of local coordinates $\{(U_r, y_r, \varphi_r)\}_{r=1}^M$ such that $\{U_r\}_{r=1}^M$ forms an open covering of $\Gamma$, $y_r = (y_r', y_{rN})$ is a rotated coordinate of $x$, and $\varphi_r: \Delta_r \to \mathbb R$ gives a graph representation $\Phi_r(y_r') := (y_r', \varphi_r(y_r'))$ of $\Gamma \cap U_r$, where $\Delta_r$ is an open cube in $\mathbb R^{N-1}_{y_r'}$. 

For $S\in \mathcal S_h$, we may assume that $S \cup \pi(S)$ is contained in some $U_r$, where $\pi : \Gamma(\delta_0) \to \Gamma$ is the projection to $\Gamma$ given in \sref{sec2.2}.
Let $b_r : \mathbb R^N \to \mathbb R^{N-1}; y_r \mapsto y_r'$ be a projection to the base set and let $S' := b_r(\pi(S))$.
Then $\Phi_r$ and $\Phi_{hr} := \pi^*\circ\Phi_r$, where $\pi^* : \Gamma\to\Gamma_h$ is the inverse map of $\pi|_{\Gamma_h}$, give smooth parameterizations of $\pi(S)$ and $S$ respectively, with the domain $S'$.
We also recall that $\pi^*$ is also written as $\pi^*(\Phi_r(y_r')) = \Phi_r(y_r') + t^*(\Phi_r(y_r')) n(\Phi_r(y_r'))$.

Let us represent integrals associated with $S$ in terms of local coordinates.
In what follows, we omit the subscript $r$ for simplicity.
First, surface integrals along $\pi(S)$ and $S$ are expressed as
\begin{align*}
	\int_{\pi(S)} f \, d\gamma = \int_{S'} f(\Phi(y')) \sqrt{\operatorname{det} G(y')} \, dy', \qquad \int_{S} f \, d\gamma_h = \int_{S'} f(\Phi_h(y')) \sqrt{\operatorname{det} G_h(y')} \, dy',
\end{align*}
where $G$ and $G_h$ denote the Riemannian metric tensors obtained from the parameterizations $\Phi$ and $\Phi_h$, respectively.
Next, let $\pi(S,\delta) := \{\bar x + tn(\bar x) \,:\, \bar x \in S, \; -\delta\le t\le \delta \}$ be a tubular neighborhood with the base $\pi(S)$, where $\delta = C_{0E}h^2$, and consider volume integrals over $\pi(S, \delta)$.
For this we introduce a one-to-one transformation $\Psi: S'\times [-\delta, \delta] \to \pi(S, \delta)$ by
\begin{equation*}
	y = \Psi(z', t) := \Phi(z') + t n(\Phi(z')) \Longleftrightarrow z' = b(\pi(y)), \; t = d(y).
\end{equation*}
Then, by change of variables, we obtain
\begin{equation*}
	\int_{\pi(S,\delta)} f(y) \, dy = \int_{S'\times [-\delta, \delta]} f(\Psi(z', t)) \operatorname{det} J(z', t) \, dz' dt,
\end{equation*}
where $J := \nabla_{(z', t)} \Psi$ denotes the Jacobi matrix of $\Psi$.
In the formulas above, $\operatorname{det}G$, $\operatorname{det}G_h$, and $\operatorname{det}J$ can be bounded, from above and below, by positive constants depending on the $C^{1,1}$-regularity of $\Omega$, provided $h$ is sufficiently small (for the proof, see \cite[Section 8]{KOZ16}).

\subsection{Proof of \eref{eq: boundary-skin estimate 2}}
In \cite[Theorem 8.3]{KOZ16}, we estimated the $L^p$-norm of a function in the full layer $\Gamma(\delta)$.
By slightly modifying the proof there, we can estimate it in $\Omega_h\setminus\Omega$, which is important to dispense with extensions from $\Omega_h$ to $\tilde\Omega$.
\begin{lem} \label{lem: proof of (2.2)}
	Let $f \in W^{1,p}(\Omega_h) \, (1\le p\le \infty)$ and $\delta = C_{0E}h^2$.  Then we have
	\begin{equation*}
		\|f\|_{L^p(\Omega_h\setminus\Omega)} \le C( \delta^{1/p} \|f\|_{L^p(\Gamma_h)} + \delta \|(n\circ\pi) \cdot \nabla f\|_{L^p(\Omega_h\setminus\Omega)} ),
	\end{equation*}
	where $C$ is independent of $\delta$ and $f$.
\end{lem}
\begin{proof}
	To simplify the notation we use the abbreviation $t^*(z')$ to imply $t^*(\Phi(z'))$.
	For each $S \in \mathcal S_h$ we observe that
	\begin{align*}
		\int_{(\Omega_h\setminus\Omega) \cap \pi(S,\delta)} |f(y)|^p\,dy &= \int_{S'} \int_0^{\max\{0, t^*(z')\}} |f(\Psi(z', t))|^p \operatorname{det} J \, dt \, dz' \\
			&\le C \int_{S'} \int_0^{\max\{0, t^*(z')\}} \Big( |f(\Phi_h(z'))|^p + |f(\Psi(z', t)) - f(\Phi_h(z'))|^p \Big) \, dt \, dz' \\
			&=: I_1 + I_2,
	\end{align*}
	and that for $z' \in S'$ and $0\le t\le t^*(z')$
	\begin{align*}
		|f(\Psi(z', t)) - f(\Phi_h(z'))| &= \left| \int_t^{t^*(z')} n(\Phi(z'))\cdot \nabla f(\Psi(z', s)) \, ds \right| \le \int_0^{t^*(z')} |n(\Phi(z')) \cdot \nabla f(\Psi(z', s))| \, ds \\
			&\le t^*(z')^{1- 1/p} \left( \int_0^{t^*(z')} |n(\Phi(z')) \cdot \nabla f(\Psi(z', s))|^p \, ds \right)^{1/p}.
	\end{align*}
	Then it follows that
	\begin{equation*}
		I_1 \le C\|t^*\|_{L^\infty(S)} \int_{S'} |f(\Phi_h(z'))|^p\, dz' \le C\delta \int_{S'} |f(\Phi_h(z'))|^p \sqrt{\operatorname{det}G_h}\,dz' = C\delta \|f\|_{L^p(S)}^p
	\end{equation*}
	and that
	\begin{align*}
		I_2 &\le C\|t^*\|_{L^\infty(S)}^p \int_{S'} \int_0^{\max\{0, t^*(z')\}} |n(\Phi(z')) \cdot \nabla f(\Psi(z', s))|^p \operatorname{det} J \, dt \, dz' \\
			&\le C\delta^p \|n\circ\pi \cdot \nabla f\|_{L^p(\pi(S,\delta))}^p.
	\end{align*}
	Adding up the above estimates for $S \in \mathcal S_h$ gives the conclusion.
\end{proof}

\begin{lem}
	For a measurable set $D \subset \mathbb R^N$ and $f \in W^{1,\infty}(\Gamma(\delta))$ we have
	\begin{equation*}
		\|f - f\circ\pi\|_{L^\infty(\Gamma(\delta) \cap D)} \le \delta \|\nabla f\|_{L^\infty(\Gamma(\delta) \cap D_{2\delta})},
	\end{equation*}
	where $D_{2\delta} = \{x\in\mathbb R^N \,:\, \operatorname{dist}(x, D) \le 2\delta\}$.
\end{lem}
\begin{proof}
	This is an easy consequence of the Lipschitz continuity of $f$.
\end{proof}

\subsection{Proof of \pref{prop: stability of extension}}
Let us prove stability properties of the extension operator $P$ defined in \sref{sec2.3}.
\begin{thm} \label{thm: extension operator}
	Let $f \in W^{k,p}(\Omega)$ with $k=0,1,2$, and $p \in [1, \infty]$.
	Then we have
	\begin{align*}
		\|Pf\|_{W^{k, p}(\Gamma(\delta))} &\le C\|f\|_{W^{k, p}(\Omega \cap \Gamma(2\delta))},
	\end{align*}
	where $C$ is independent of $\delta$ and $f$.
\end{thm}
\begin{proof}	
	First, for each $S \in \mathcal S_h$ we show
	\begin{equation*}
		\|Pf\|_{L^p(\pi(S,\delta) \setminus \Omega)}^p \le C \|f\|_{L^p(\pi(S, 2\delta) \cap \Omega)}^p.
	\end{equation*}
	In fact we have
	\begin{align*}
		\int_{\pi(S,\delta) \setminus \Omega} |Pf(y)|^p \, dy &\le C \int_{S' \times [0, \delta]} |3f(z' - tn(z')) - 2f(z' - 2t n(z'))|^p \,dz'dt \\
			&\le C \int_{S' \times [0, \delta]} \Big( |f(z' - tn(z'))|^p + |f(z' - 2t n(z'))|^p \Big) \,dz'dt \\
			&\le C \int_{\pi(S, \delta) \cap \Omega} |f(y)|^p \,dy + C \int_{\pi(S, 2\delta) \cap \Omega} |f(y)|^p \,dy.
	\end{align*}
	
	Next we show
	\begin{equation} \label{eq: stability of grad Pf}
		\|\nabla Pf\|_{L^p(\pi(S,\delta) \setminus \Omega)}^p \le C \|\nabla f\|_{L^p(\pi(S, 2\delta) \cap \Omega)}^p.
	\end{equation}
	Since by the chain rule $\nabla_y = \nabla_y(b\circ\pi)\nabla_{z'} + (\nabla_yd)\partial_t$ and since $Pf(y) = 3f\circ\Psi(z', -t) - 2f\circ\Psi(z', -2t)$, it follows that
	\begin{align}
		\nabla Pf(y) &= \nabla_y(b\circ\pi) \Big( 3\nabla_{z'}(f\circ\Psi)|_{(z',-t)} - 2\nabla_{z'}(f\circ\Psi)|_{(z',-2t)} \Big) \notag \\
			&\hspace{8.5mm} + \nabla_yd \Big( -3\partial_t(f\circ\Psi)|_{(z',-t)} + 4\partial_t(f\circ\Psi)|_{(z',-2t)} \Big), \quad y \in \pi(S,\delta) \setminus \Omega. \label{eq: representation of grad Pf}
	\end{align}
	In particular, if $y \in \Gamma$ i.e.\ $t = 0$, then
	\begin{equation*}
		\nabla Pf(y) = \nabla_y(b\circ\pi) \nabla_{z'}(f\circ\Psi)|_{(z',0)} + (\nabla_yd) \partial_t(f\circ\Psi)|_{(z', 0)} = J^{-1}(z', 0) J(z', 0)\nabla_yf(y) = \nabla f(y),
	\end{equation*}
	which ensures that $Pf(y) \in W^{2, p}(\pi(S, \delta))$.
	Now, noting that $\nabla_y \begin{spmatrix} b\circ\pi \\ d \end{spmatrix} = J^{-1}(z', t)$ and that $\nabla_{(z', t)}(f\circ\Psi)|_{(z', -it)} = J(z', -it) (\nabla_y f)|_{\Psi(z', -it)} \, (i=1,2)$ where $J$ and $J^{-1}$ depend on the $C^{1,1}$-regularity of $\Omega$, we deduce that
	\begin{equation*}
		\int_{\pi(S,\delta) \setminus \Omega} |\nabla Pf(y)|^p \, dy \le C \int_{S'\times [0, \delta]} \Big( \big| (\nabla_yf)|_{\Psi(z', -t)} \big|^p + \big| (\nabla_yf)|_{\Psi(z', -2t)} \big|^p \Big) \, dz'dt,
	\end{equation*}
	from which \eref{eq: stability of grad Pf} follows.
	
	Finally we show
	\begin{equation} \label{eq: stability of grad-square Pf}
		\|\nabla^2 Pf\|_{L^p(\pi(S,\delta) \setminus \Omega)}^p \le C (\|\nabla^2 f\|_{L^p(\pi(S, 2\delta) \cap \Omega)}^p + \|\nabla f\|_{L^p(\pi(S, 2\delta) \cap \Omega)}^p).
	\end{equation}
	By differentiating \eref{eq: representation of grad Pf} we find that for $y \in \pi(S, \delta) \setminus \Omega$
	\begin{equation*}
		\nabla^2Pf(y) = \sum_{i=1}^2 \Big( A_i(z', t) \nabla_{(z', t)}^2(f\circ\Psi)_{(z', -it)} + B_i(z', t) \nabla_{(z', t)}(f\circ\Psi)_{(z', -it)} \Big),
	\end{equation*}
	where the coefficient tensors $A_i, B_i$ depend on the $C^{1,1}$-regularity of $\Omega$.
	Then the $L^p$-norm of the above quantity can be estimated similarly as before and one obtains \eref{eq: stability of grad-square Pf}.
	
	Adding up the above estimates for $S \in \mathcal S_h$ deduces the desired stability properties.
\end{proof}

We also need local stability of the extension operator as follows.
\begin{cor} \label{cor: local stability of P}
	For a measurable set $D \subset \mathbb R^N$ and $\delta = C_{0E}h^2$ we have
	\begin{equation*}
		\|Pf\|_{W^{k, \infty}(\Gamma(\delta) \cap D)} \le C\|f\|_{W^{k, \infty}(\Omega \cap \Gamma(2\delta) \cap D_{3\delta})} \quad (k=0,1,2),
	\end{equation*}
	where $D_{3\delta} = \{x\in\mathbb R^N \,:\, \operatorname{dist}(x, D) \le 3\delta\}$ and $C$ is independent of $\delta$, $f$, and $D$.
\end{cor}
\begin{proof}
	We address the $L^\infty$-norm of $\nabla Pf$; the treatment of $Pf$ and $\nabla^2Pf$ is similar.
	For each $S \in \mathcal S_h$, we find from the analysis of \tref{thm: extension operator} that $\nabla Pf(y)$ for $y \in \pi(S, \delta) \setminus \Omega$ can be expressed as
	\begin{equation*}
		\nabla Pf(y) = \sum_{i=1}^2 A_i(z', t) (\nabla_yf)|_{\Psi(z', -it)},
	\end{equation*}
	where the matrices $A_i$ depend on the $C^{0,1}$-regularity of $\Omega$.
	Then the desired estimate follows from the observation that if $y = \Psi(z', t) \in \pi(S, \delta) \cap D \setminus \Omega$ then $\Psi(z', -it) \in \pi(S, i\delta) \cap D_{3\delta} \cap \Omega$ for $i=1,2$.
\end{proof}

\section{Analysis of regularized Green's functions}
\subsection{Estimates for $\tilde g$}
Recall that for arbitrarily fixed $x_0 \in \Omega_h$ we have introduced $\eta \in C_0^\infty(\Omega_h \cap \Omega)$ and $g_m \in C^\infty(\overline\Omega) \, (m=0,1)$ in \sref{sec3}.
Using the Green's function $G(x, y)$ for the operator $-\Delta + 1$ in $\Omega$ with the homogeneous Neumann boundary condition, one can represent $g_m$ as
\begin{equation*}
	g_0(x) = \int_{\operatorname{supp}\eta} G(x, y)\eta(y) \, dy, \qquad g_1(x) = -\int_{\operatorname{supp}\eta} \partial_yG(x, y)\eta(y) \, dy, \qquad x \in \Omega.
\end{equation*}
The following derivative estimates for $G$ are well known (see e.g.\ \cite[p.\ 965]{Kra67b}):
\begin{equation*}
	|\nabla_x^k \nabla_y^l G(x, y)| \le \begin{cases} C(1 + |x - y|^{2 - l - k - N}) &(l+k+N > 2) , \\ C(1 + \big|\log|x - y|\big|) & (N=2, l=k=0). \end{cases}
\end{equation*}
From this, combined with a dyadic decomposition of $\Omega$, we derive some local and global estimates for $g_m$ and its extension $\tilde g_m := Pg_m$.
Below the subscript $m$ will be dropped for simplicity.

\begin{lem} \label{lem: local Linf estimate of g}
	Let $\mathcal A_\Omega(x_0, d_0) = \{\Omega \cap A_j\}_{j=0}^J$ be a dyadic decomposition of $\Omega$ with $d_0 \in [4h, 1]$.
	Then, for $j=1, \dots, J$ and $k\ge 0$ we have
	\begin{equation*}
		\|\nabla^k g\|_{L^\infty(\Omega\cap A_j)} \le \begin{cases} C(1 + d_j^{2-m-k-N}) & (m+k+N>2), \\ C(1 + |\log d_j|) & (N=2, m=k=0), \end{cases} \qquad
	\end{equation*}
	where $C$ is independent of $x_0, d_0, h, j$, and $\partial$.
\end{lem}
\begin{proof}
	We only consider $m+k+N>2$ because the other case can be treated similarly.
	Notice that if $x \in \Omega\cap A_j \, (j\ge1)$ and $y \in \operatorname{supp}\eta$ then $|x - y| \ge \frac34 d_{j-1}$, which is obtained from $|x - x_0| \ge d_{j-1}$ and $|y - x_0|\le h$.
	It then follows that
	\begin{align*}
		\|\nabla^k g\|_{L^\infty(\Omega\cap A_j)} &= \sup_{x\in \Omega \cap A_j} \left| \int_{\operatorname{supp} \eta} \partial_y^m \nabla_x^kG(x, y) \eta(y) \, dy \right| \le \sup_{|x - y| \ge \frac34 d_{j-1}} |\partial_y^m\nabla_x^kG(x, y)| \\
			&\le C(1 + d_j^{2-m-k-N}),
	\end{align*}
	which completes the proof.
\end{proof}

We transfer these estimates in $\Omega$ to those in $\tilde\Omega = \Omega\cup\Gamma(\delta)$ using an extension operator and its stability.
\begin{lem} \label{lem: local Lp estimate of g}
	Let $\mathcal A_\Omega(x_0, d_0) = \{\Omega \cap A_j\}_{j=0}^J$ be a dyadic decomposition of $\Omega$ with $d_0 \in [h, 1]$, $\delta = C_{0E}h^2$.
	For $p \in [1,\infty]$, $j = 1, \dots, J$, and $m = 0,1$, we have
	\begin{equation*}
		\|\nabla^2 \tilde g\|_{L^p(\tilde\Omega \cap A_j)} \le Cd_j^{-m-N/p'},
	\end{equation*}
	where $p' = p/(p-1)$ and $C$ is independent of $x_0, d_0, h, j$, and $\partial$.
\end{lem}
\begin{proof}
	By the H\"older inequality and \lref{lem: local Linf estimate of g} we see that
	\begin{align*}
		\|\nabla^2 \tilde g\|_{L^p(\tilde\Omega \cap A_j)} &\le C |\Omega_h \cap A_j|^{1/p} \|\nabla^2 \tilde g\|_{L^\infty(\tilde\Omega \cap A_j)} \le Cd_j^{N/p} \|g\|_{W^{2, \infty}(\Omega \cap A_j^{(1/4)})} \\
			&\le Cd_j^{N/p} (1 + d_j^{2-m-N} + d_j^{1-m-N} + d_j^{-m-N}) \le Cd_j^{-m-N/p'},
	\end{align*}
	where we have used $d_j \le 2 \operatorname{diam}\Omega$ in the last inequality.
\end{proof}

We also need local estimates in intersections of annuli and boundary-skins (or boundaries).
\begin{lem} \label{lem: local Lp estimate of g in boundary-skins}
	Let $\mathcal A_\Omega(x_0, d_0) = \{\Omega \cap A_j\}_{j=0}^J$ be a dyadic decomposition of $\Omega$ with $d_0 \in [h, 1]$, $\delta = C_{0E}h^2$.
	For $p \in [1, \infty]$, $j = 1, \dots, J$, $m = 0,1$, and $k=0,1,2$, we have
	\begin{align*}
		\|\nabla^k \tilde g\|_{L^p(\Gamma(\delta) \cap A_j)} &\le C (\delta d_j^{N-1})^{1/p} (1 + d_j^{2 - m - k - N}), \\
		\|\nabla^k g\|_{L^p(\Gamma \cap A_j)} + \|\nabla^k \tilde g\|_{L^p(\Gamma_h \cap A_j)} &\le C d_j^{(N-1)/p} (1 + d_j^{2 - m - k - N}),
	\end{align*}
	provided $m+k+N > 2$.  Even when $N=2$ and $m=k=0$, the above estimates hold with the factor $d_j^{2-m-k-N}$ replaced by $|\log d_j|$.
	The constants $C$ are independent of $x_0, d_0, h, j$, and $\partial$.
\end{lem}
\begin{proof}
	We only consider $m+k+N > 2$ since the other case may be treated similarly.
	From \cref{cor: local stability of P} and \lref{lem: local Linf estimate of g} we deduce that (note that $(A_j)_{3\delta} \subset A_j^{(1/4)}$ for small $h$)
	\begin{align*}
		\|\nabla^k \tilde g\|_{L^p(\Gamma(\delta) \cap A_j)} &\le |\Gamma(\delta) \cap A_j|^{1/p} \|\nabla^k \tilde g\|_{L^\infty(\Gamma(\delta) \cap A_j)} \le C(\delta d_j^{N-1})^{1/p} \|g\|_{W^{k,\infty}(\Omega \cap \Gamma(2\delta) \cap A_j^{(1/4)})} \\
			&\le C(\delta d_j^{N-1})^{1/p} (1 + d_j^{2-m-k-N}),
	\end{align*}
	where we have used $d_j \le 2 \operatorname{diam}\Omega$ in the second line.
	Similarly,
	\begin{align*}
		\|\nabla^k \tilde g\|_{L^p(\Gamma_h \cap A_j)} &\le |\Gamma_h \cap A_j|^{1/p} \|\nabla^k \tilde g\|_{L^\infty(\Gamma_h \cap A_j)} \le Cd_j^{(N-1)/p} \|g\|_{W^{k,\infty}(\Omega \cap \Gamma(2\delta) \cap A_j^{(1/4)})} \\
			&\le Cd_j^{(N-1)/p}(1 + d_j^{2-m-k-N}).
	\end{align*}
	One sees that $\|\nabla^k g\|_{L^p(\Gamma\cap A_j)}$ obeys the same estimate.
\end{proof}

\begin{rem}
	The three lemmas above remain true with $A_j$ replaced by $A_j^{(s)} \, (0\le s < 1)$, where the constants $C$ become dependent on the choice of $s$.
\end{rem}

Especially when $p=1$, the following global estimate in a boundary-skin layer holds.
\begin{cor} \label{cor: boundary-skin estimates for g}
	Let $\delta = C_{0E}h^2$ with sufficiently small $h$.  Then we have
	\begin{equation*}
		\|\tilde g_0\|_{W^{k, 1}(\Gamma(\delta))} \le \begin{cases} C\delta & (k=0), \\ C\delta|\log h| & (k=1), \\ C\delta h^{-1} & (k=2), \end{cases} \qquad
		\|\nabla^k g_0\|_{L^1(\Gamma)} + \|\nabla^k \tilde g_0\|_{L^1(\Gamma_h)} \le \begin{cases} C & (k=0), \\ C|\log h| & (k=1), \\ Ch^{-1} & (k=2), \end{cases}
	\end{equation*}
	and
	\begin{equation*}
		\|\tilde g_1\|_{W^{k, 1}(\Gamma(\delta))} \le \begin{cases} C\delta|\log h| & (k=0), \\ C\delta h^{-1} & (k=1), \\ C\delta h^{-2} & (k=2), \end{cases} \qquad
		\|\nabla^k g_1\|_{L^1(\Gamma)} + \|\nabla^k \tilde g_1\|_{L^1(\Gamma_h)} \le \begin{cases} C|\log h| & (k=0), \\ Ch^{-1} & (k=1), \\ Ch^{-2} & (k=2), \end{cases}
	\end{equation*}
	where $C$ is independent of $x_0, h$, and $\partial$.
\end{cor}
\begin{proof}
	We only consider the estimates in $W^{k,1}(\Gamma(\delta))$ because the boundary estimates can be derived similarly.
	With a dyadic decomposition $\mathcal A_\Omega(x_0, 4h) = \{\Omega \cap A_j\}_{j=0}^J$, we compute $\sum_{j=0}^J \|\tilde g\|_{W^{k,1}(\Gamma(\delta) \cap A_j)}$.
	When $j \ge 1$, it follows from \lref{lem: local Lp estimate of g in boundary-skins} that
	\begin{equation} \label{eq1: Cor B.1}
		\|\tilde g\|_{W^{k, 1}(\Gamma(\delta) \cap A_j)} \le 
		\begin{cases}
			C(\delta d_j^{N-1}) d_j^{2-m-k-N} & (m+k+N>2), \\
			C(\delta d_j^{N-1}) |\log d_j| & (N=2, m=k=0).
		\end{cases}
	\end{equation}
	When $j = 0$, notice that $\operatorname{dist}(\operatorname{supp} \eta, \Gamma(2\delta)) \ge Ch = \frac{C}4 d_0$ for sufficiently small $h$, which results from \eref{eq: properties of eta}.
	Then, calculating in the same way as above, we find that \eref{eq1: Cor B.1} holds for $j=0$ as well.
	Adding up the above estimate for $j = 0, \dots, J$ and using \eref{eq: geometric sequence}, we obtain the desired result.
\end{proof}
\begin{rem}
	We could improve the above estimates for $g_0$ when $k=1$ if the Dirichlet boundary condition were considered.
	In fact, the Green's function $G_D(x, y)$ in this case is known to satisfy $|\nabla_x G_D(x, y)| \le C\operatorname{dist}(y, \partial\Omega) |x - y|^{-N}$ (see \cite[Theorem 3.3(v)]{GrWi82}).
	Then, taking a dyadic decomposition with $d_0 = \operatorname{dist}( \operatorname{supp}\eta, \partial\Omega ) \ge Ch$, we see that
	\begin{equation*}
		\|\nabla\tilde g_0\|_{L^1(\Gamma_h)} \le C\sum_{j=0}^J d_j^{N-1} \|\nabla \tilde g_0\|_{L^\infty(\Gamma_h \cap A_j)}
		\le C \operatorname{dist}( \operatorname{supp}\eta, \partial\Omega ) \sum_{j=0}^J d_j^{-1} \le Cd_0 d_0^{-1} = C,
	\end{equation*}
	and that $\|\nabla\tilde g_0\|_{L^1(\Gamma(\delta))} \le C\delta$.
	However, such an auxiliary Green's function estimate is not available in the case of the Neumann boundary condition.
	A similar inequality is proved in \cite[eq.\ (5.8)]{ScWa82} by a different method using the maximum principle, but its extension to the Neumann case seems non-trivial.
\end{rem}

\subsection{Estimates for $\tilde w$}
Let us recall the situation of \sref{sec5}: fixing a dyadic decomposition $\mathcal A_\Omega(x_0, d_0)$ and an annulus $A_j \, (0\le j\le J)$, we have introduced the solution $w \in C^\infty(\overline\Omega)$ of \eref{eq: PDE for w} for arbitrary $\varphi \in C_0^\infty(\Omega_h \cap A_j)$ such that $\|\varphi\|_{L^2(\Omega_h \cap A_j)} = 1$.
Hence $w$ is represented, using the Green's function $G(x, y)$, as
\begin{equation*}
	w(x) = \int_{\Omega \cap \Omega_h \cap A_j} G(x, y) \varphi(y) \, dy \qquad (x \in\Omega).
\end{equation*}
Then we obtain the following local $L^\infty$-estimates away from $A_j$:

\begin{lem} \label{lem: estimate for w}
	For $k=0,1,2$ and $\delta = C_{0E}h^2$, we have
	\begin{equation*}
		\|\tilde w\|_{W^{k,\infty}(\tilde\Omega \setminus A_j^{(1/2)})} \le \begin{cases} Cd_j^{2-k-N/2} & (N+k>2), \\ Cd_j (1+|\log d_j|) & (N=2, k=0), \end{cases}
	\end{equation*}
	where $\tilde\Omega := \Omega\cup\Gamma(\delta)$, $\tilde w =: Pw$, and $C$ is independent of $h,x_0,d_0$, and $j$.
\end{lem}
\begin{proof}
	We focus on the case $N+k>2$; the other case is similar.
	We find that
	\begin{align*}
		\|\tilde w\|_{W^{k,\infty}(\tilde\Omega \setminus A_j^{(1/2)})} &\le C\|w\|_{W^{k,\infty}(\Omega \setminus A_j^{(1/4)})}
		= C\sum_{l=0}^k \sup_{x \in \Omega \setminus A_j^{(1/4)}} \left| \int_{\Omega \cap \Omega_h \cap A_j} \nabla_x^l G(x, y) \varphi(y) \, dy \right| \\
			&\le C \sum_{l=0}^k |\Omega \cap \Omega_h \cap A_j|^{1/2} \sup_{|x - y| \ge d_{j-1}/8} |\nabla_x^l G(x, y)| \, \|\varphi\|_{L^2(\Omega \cap \Omega_h \cap A_j)} \\
			&\le Cd_j^{N/2} (1 + d_j^{2-N} + \cdots + d_j^{2-k-N}) \le Cd_j^{2-k-N/2},
	\end{align*}
	where we have used $d_j \le 2 \operatorname{diam}\Omega$ in the last inequality.
\end{proof}
\begin{rem}
	The lemma remains true with $A_j^{(1/2)}$ replaced by $A_j^{(s)} \, (0<s\le 1)$, where the constant $C$ becomes dependent on the choice of $s$.
\end{rem}



\end{document}